\documentclass{amsart}

\usepackage{amsfonts}
\usepackage{amsmath}
\usepackage{amssymb}
\usepackage{latexsym}
\usepackage{amscd}
\usepackage{amsthm}
\usepackage[dvips]{graphicx}
\usepackage{subfigure}
\usepackage{bm}

\newcommand{\Z}{\mathbb{Z}}
\newcommand{\M}{\mathcal{M}}
\newcommand{\Hy}{\mathcal{H}}
\newcommand{\A}{\mathcal{A}}
\newcommand{\la}{\langle}
\newcommand{\ra}{\rangle}
\newcommand{\1}{a_1}
\newcommand{\2}{a_2}
\newcommand{\3}{a_3}
\newcommand{\4}{a_4}
\newcommand{\5}{a_5}
\newcommand{\6}{b_1}
\newcommand{\7}{b_2}
\newcommand{\8}{b_3}
\newcommand{\9}{b_4}
\newcommand{\0}{b_5}

\newtheorem{thm}{Theorem}[section]
\newtheorem{prop}[thm]{Proposition}

\address{Department of Mathematics \endgraf 
Faculty of Science and Technology \endgraf 
Tokyo University of Science \endgraf 
Noda, Chiba, 278-8510, Japan \endgraf
\it{e-mail address : \rm{kobayashi\_ryoma@ma.noda.tus.ac.jp}}}

\begin{document}

\title{On genera of Lefschetz fibrations and finitely presented groups}
\author{Ryoma Kobayashi}
\maketitle

\begin{abstract}
It is known that every finitely presented group is the fundamental group
 of the total space of a Lefschetz fibration. 
In this paper, we give another proof which improves the result of
 Korkmaz. 
In addition, Korkmaz defined the genus of a finitely presented group. 
We also evaluate upper bounds for genera of some finitely
 presented groups.
\end{abstract}

\section{Introduction}

Gompf \cite{g} proved that every finitly presented group is the
fundamental group of a closed symplectic $4$-manifold.
Donaldson \cite{d} proved that every closed
symplectic $4$-manifold admits a Lefschetz pencil.
By blowing up the base locus of a Lefschetz pencil, we obtain a Lefschetz
fibration over $S^2$.
In addition, blowing up does not change the fundamental group of a
$4$-manifold.
Therefore, it immediately follows that every finitely presented group
is the fundamental group of the total space of a Lefschetz fibration. 

Amoros-Bogomolov-Katzarkov-Pantev \cite{a} and Korkmaz \cite{k2}
also constructed Lefschetz fibrations whose fundamental groups are a
given finitely presented group.
In particular, Korkmaz \cite{k2} provided explicitly a genus and a
monodromy of such a Lefschetz fibration.

Let $F_n=\la{}g_1,\dots,g_n\ra$ be the free group of rank $n$.
For $x\in{}F_n$, the {\it syllable length} $\ell(x)$ of $x$ is defined
by
$$\ell(x)=\min\{s\mid{}x=g_{i(1)}^{m(1)}\cdots{}g_{i(s)}^{m(s)}\}.$$
For a finitely presented group $\Gamma$ with a presentation 
$\Gamma=\la{}g_1,\dots,g_n\mid{}r_1,\dots,r_k\ra$, Korkmaz \cite{k2}
proved that for any
$\displaystyle{}g\geq2(n+\sum_{1\leq{}i\leq{}k}\ell(r_i)-k)$ there
exists a genus-$g$ Lefschetz fibration $f:X\to{}S^2$ such that the
fundamental group $\pi_1(X)$ is isomorphic to $\Gamma$, providing
explicitly a monodromy.

In this paper, we improve this result. 

\begin{thm}\label{1.1}
Let $\Gamma$ be a finitely presented group with a presentation 
$\Gamma=\la{}g_1,\dots,g_n\mid{}r_1,\dots,r_k\ra$, and let 
$\displaystyle\ell=\max_{1\leq{}i\leq{}k}\{\ell(r_i)\}$.  
Then for any $g\geq2n+\ell-1$, there exists a genus-$g$ Lefschetz
 fibration $f:X\to{}S^2$ such that the fundamental group $\pi_1(X)$ is
 isomorphic to $\Gamma$.
\end{thm}

In this theorem,
if $k=0$, we suppose $\ell=1$.
We will prove the theorem by providing an explicit monodromy.

In addition, Korkmaz \cite{k2} defined the {\it genus} $g(\Gamma)$ of a
finitely presented group $\Gamma$ to be the minimal genus of a Lefschetz
fibration whose fundamental group is isomorphic to $\Gamma$. 
It immediately follows from the above theorem that the definition of the
genus of a finitely presented group is well-defined.

We will also prove the following theorem.

\begin{thm}\label{1.2}
\begin{enumerate}
 \item Let $B_n$ denote the $n$-strands braid group.
       Then for $n\geq3$, we have $2\leq{}g(B_n)\leq4$ and $g(B_2)=1$.
 \item Let $\Hy_g$ be the hyperelliptic mapping class group of a closed
       connected orientable surface of genus $g\geq1$.
       Then we have $2\leq{}g(\Hy_g)\leq4$.
 \item Let $\M_{0,n}$ denote the mapping class group of a sphere with
       $n$ punctures.
       Then for $n\geq3$, we have $2\leq{}g(\M_{0,n})\leq4$ and
       $g(\M_{0,2})=2$.
 \item Let $S_n$ denote the $n$-symmetric group.
       Then for $n\geq3$, we have $2\leq{}g(S_n)\leq4$ and $g(S_2)=2$.
 \item Let $\A_n$ denote the $n$-Artin group associated to the Dynkin
       diagram shown in Figure~\ref{artin}. 
       Then for $n\geq5$, we have $2\leq{}g(\A_n)\leq5$.
 \item Let $n,k\geq0$ be integers with $n+k\geq3$, and let
       $m_1,\dots,m_k\geq2$ be integers.
       Then we have
       $\frac{n+k+1}{2}\leq{}g(\Z^n\oplus\Z_{m_1}\oplus\cdots\oplus\Z_{m_k})\leq{}n+k+1$.
\end{enumerate}
\end{thm}

\begin{figure}[h]
\includegraphics[scale=0.5]{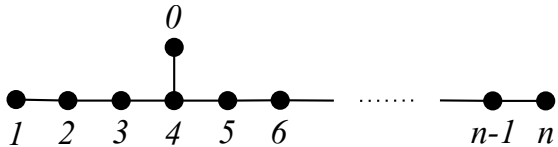}
\caption{The Dynkin diagram.}\label{artin}
\end{figure}

\section{A Lefschetz fibration and preliminaries}

\subsection{A Lefschetz fibration and its monodromy}

Here, we review briefly the theory of Lefschetz
fibrations.

Let $X$ be a closed connected orientable smooth $4$-manifold.
A smooth map $f:X\to{}S^2$ is a genus-$g$ {\it Lefschetz fibration} over
$S^2$ if it satisfies following properties:
\begin{itemize}
 \item All regular fibers are diffeomorphic to a closed connected
       oriented surface of genus $g$.
 \item Each critical point of $f$ has an orientation-preserving chart on
       which $f(z_1,z_2)=z_1^2+z_2^2$ relative to a suitable smooth
       chart on $S^2$.
 \item Each singular fiber contains only one critical point.
 \item $f$ is {\it relatively minimal}, that is, no fiber contains an
       embedded sphere with the self-intersection number $-1$.
\end{itemize} 

Let $\M_g$ be the mapping class group of a closed connected oriented
surface $\Sigma_g$ of genus $g$, that is, the group of isotopy classes
of orientation-preserving diffeomorphisms $\Sigma_g\to\Sigma_g$.
In this paper, for elements $x$ and $y$ of a group, the composition $xy$
means that we first apply $x$ and then $y$.
So for $f,g\in\M_g$, the composition $fg$ means that we first apply $f$
and then $g$.
For a simple closed curve $c$ on $\Sigma_g$, let $t_c$ be the isotopy
class of the right Dehn twist about $c$ (see Figure~\ref{dehn}).
For a genus-$g$ Lefschetz fibration which has $n$ singular fibers, there
are simple closed curves $c_1,\dots,c_n$ on $\Sigma_g$, each of which is
called the {\it vanishing cycle}, such that each singular fiber $F_i$ is
diffeomorphic to $\Sigma_g/{c_i}$ and $t_{c_1}\cdots{}t_{c_n}=1$. 
This equation is called the {\it monodromy} of a Lefschetz fibration. 
Conversely, if there are simple closed curves $c_1,\dots,c_n$ on
$\Sigma_g$ such that $t_{c_1}\cdots{}t_{c_n}=1$, then we can construct a
genus-$g$ Lefschetz fibration with the monodromy
$t_{c_1}\cdots{}t_{c_n}=1$. 

\begin{figure}[h]
\includegraphics[scale=1.0]{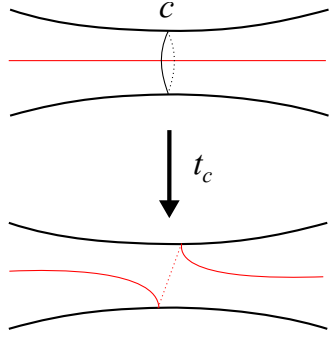}
\caption{The right Dehn twist about $c$.}\label{dehn}
\end{figure}

For a Lefschetz fibration $f:X\to{}S^2$, a smooth map $s:S^2\to{}X$ is a
section of $f$ if $f\circ{}s:S^2\to{}S^2$ is the identity map.

\begin{figure}[h]
\includegraphics[scale=0.8]{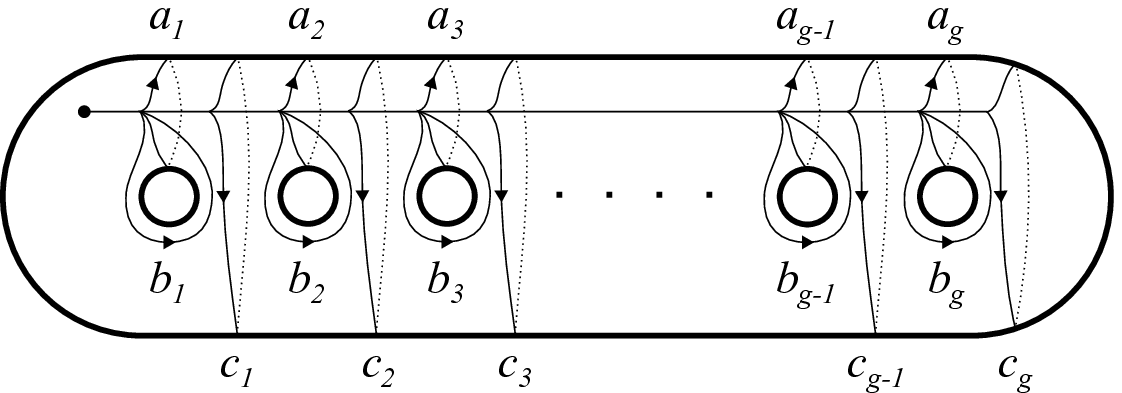}
\caption{}\label{abc}
\end{figure}

\begin{figure}[h]
\includegraphics[scale=0.5]{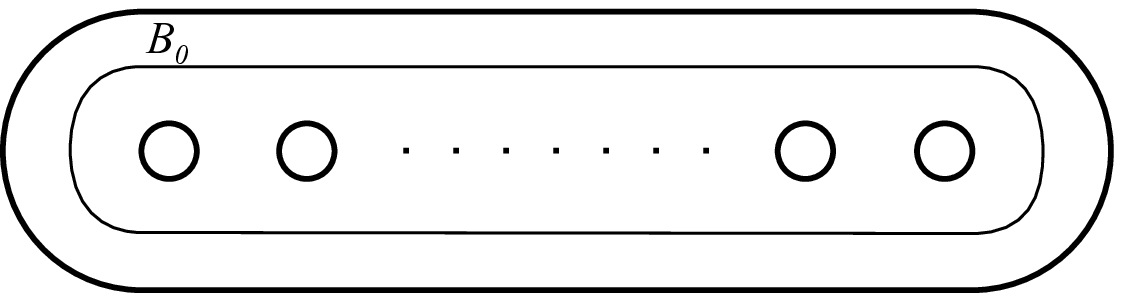}
\vspace{0.2cm}

\includegraphics[scale=0.5]{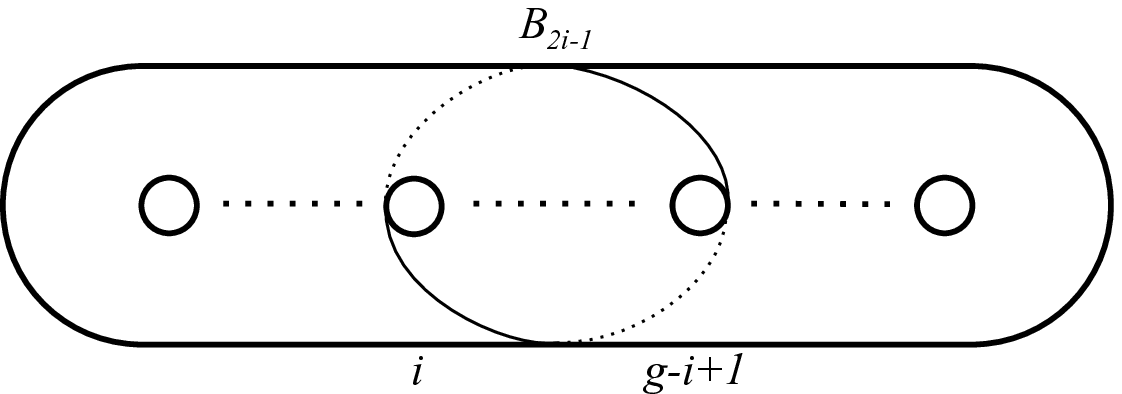}
\hspace{0.2cm}
\includegraphics[scale=0.5]{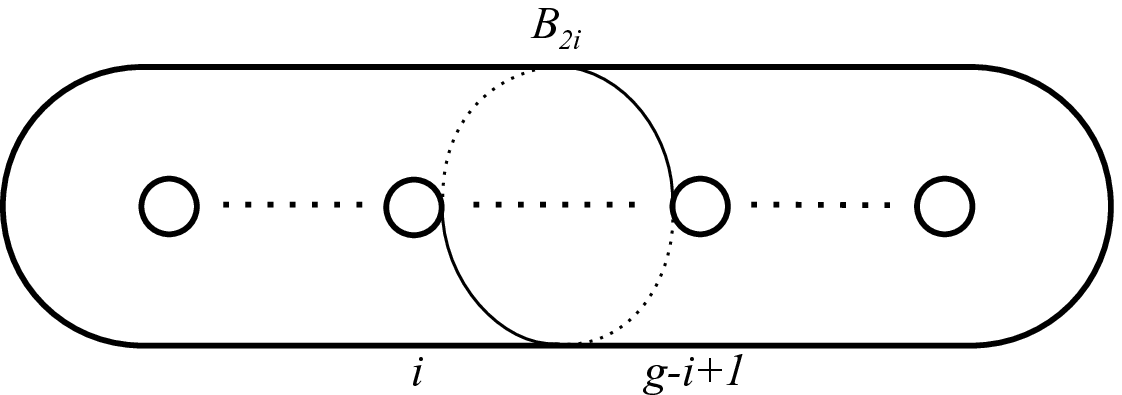}
\vspace{0.2cm}

\subfigure[The case $g$ is odd.]{\includegraphics[scale=0.5]{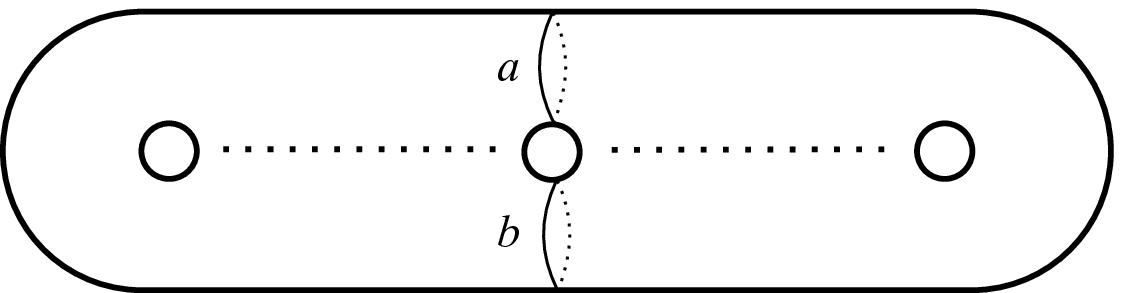}}
\hspace{0.2cm}
\subfigure[The case $g$ is even.]{\includegraphics[scale=0.5]{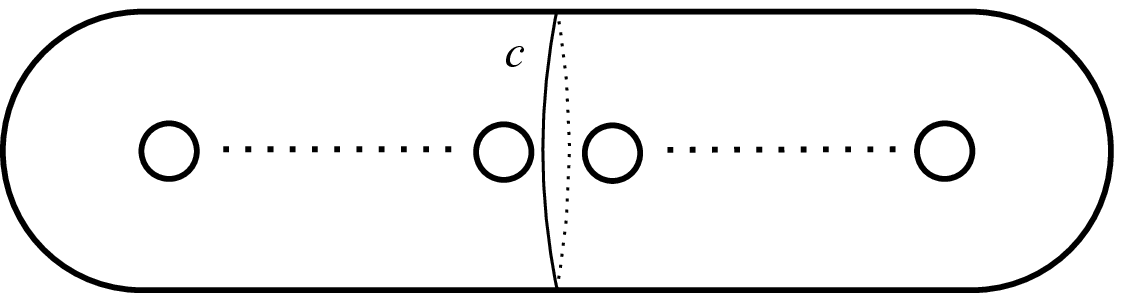}}
\caption{}\label{mkc}
\end{figure}

For a closed connected orientable surface $\Sigma_g$ of genus $g$, let
$a_1,\dots,a_g,b_1,\dots,b_g$ and $c_1,\dots,c_g$ be loops on $\Sigma_g$
as shown in Figure~\ref{abc}.
Then the fundamental group $\pi_1(\Sigma_g)$ of $\Sigma_g$ has a
following presentation
$$\pi_1(\Sigma_g)=\la{}a_1,b_1,\dots,a_g,b_g\mid{}r\ra,$$
where $r=b_g^{-1}\cdots{}b_1^{-1}(a_1b_1a_1^{-1})\cdots(a_gb_ga_g^{-1})$.

Let $B_0,\dots,B_g$ and $a,b,c$ be simple closed curves on $\Sigma_g$ as
shown in Figure~\ref{mkc}.
In this paper, let $W$ denote the following
$$
W=
\left\{
\begin{array}{ll}
(t_ct_{B_g}\cdots{}t_{B_0})^2&\textrm{when}~g~\textrm{is even}, \\
(t_a^2t_b^2t_{B_g}\cdots{}t_{B_0})^2&\textrm{when}~g~\textrm{is odd}.
\end{array}
\right.
$$
It was shown in \cite{k1} that $W=1$ in the mapping class group $\M_g$
of $\Sigma_g$.
In addition, the Lefschetz fibration $f_W:X_W\to{}S^2$ with the monodromy
$W=1$ has a section (see \cite{k1} and \cite{k2}).

\subsection{Preliminaries}

We now state the way to obtain the presentation of the fundamental group
of a Lefschetz fibration with a section.

\begin{prop}[cf. \cite{gs}]\label{2.1}
Let $f:X\to{}S^2$ be a genus-$g$ Lefschetz fibration with the monodromy
 $t_{c_1}\cdots{}t_{c_n}=1$. 
Suppose that $f$ has a section. 
Then we have
$$\pi_1(X)\cong\pi_1(\Sigma_g)/{\la{}c_1,\dots,c_n\ra},$$
where we regard $c_1,\dots,c_n$ as elements in $\pi_1(\Sigma_g)$.
\end{prop}

For $x,y\in\M_g$, let $x^y=y^{-1}xy$. 
For example, for simple closed curves $c_1,\dots,c_n$ on $\Sigma_g$ and
$h\in\M_g$, we have 
$(t_{c_1}\cdots{}t_{c_n})^h=(h^{-1}t_{c_1}h)\cdots(h^{-1}t_{c_n}h)=t_{(c_1)h}\cdots{}t_{(c_n)h}$,
where $(c_i)h$ means the image of $c_i$ by $h$.

\begin{prop}[\cite{k2}]\label{2.2}
Let $f:X\to{}S^2$ be a genus-$g$ Lefschetz fibration with the monodromy
 $V=t_{c_1}\cdots{}t_{c_n}=1$. 
Suppose that $f$ has a section. 
Let $d$ be a simple closed curve on $\Sigma_g$ which intersects some
 $c_i$ transversely at only one point.  
Let $f':X'\to{}S^2$ be the genus-$g$ Lefschetz fibration with the
 monodromy $VV^{t_d}=1$.
 Then we have
$$\pi_1(X')\cong\pi_1(\Sigma_g)/{\la c_1,\dots,c_n,d\ra},$$
where we regard $c_1,\dots,c_n$ and $d$ as elements in
 $\pi_1(\Sigma_g)$. 
\end{prop}

In this paper, we denote the Lefschetz fibration with the monodromy
$V=1$ by $f_V:X_V\to{}S^2$.
For example, in the above proposition, $f=f_V$, $X=X_V$ and
$f'=f_{VV^{t_c}}$, $X'=X_{VV^{t_c}}$.

We next state results of Korkmaz \cite{k2}.

\begin{thm}[\cite{k2}]\label{2.3}
\begin{enumerate}
 \item Let $\Sigma_g$ be a closed connected orientable surface of genus
       $g\geq0$.
       Then we have $g(\pi_1(\Sigma_g))=g$.
 \item Let $m(\Gamma)$ denote the minimal number of generators for
       $\Gamma$.
       Then we have $\frac{m(\Gamma)}{2}\leq{}g(\Gamma)$, with the
       equality if and only if $\Gamma$ is isomorphic to
       $\pi_1(\Sigma_g)$.
 \item For the mapping class group $\M_1$ of $\Sigma_1$, we have
       $2\leq{}g(\M_1)\leq4$.
 \item Let $B_n$ denote the $n$-strands braid group. Then for $n\geq3$,
       we have $2\leq{}g(B_n)\leq5$.
 \item Let $n,k\geq 0$ be integers with $n+k\geq3$, and let
       $m_1,\dots,m_k\geq2$ be integers.
       Then we have
       $\frac{n+k+1}{2}\leq{}g(\Z^n\oplus\Z_{m_1}\oplus\cdots\oplus\Z_{m_k})\leq2(n+k)+1$. 
\end{enumerate}
\end{thm}

In Theorem~\ref{1.2}, (4) and (5) of Theorem~\ref{2.3} are improved.

\section{Proof of Theorem~\ref{1.1}}

First of all, we show a proposition used in proofs of Theorem~\ref{1.1}
and \ref{1.2}. 
For elements $x$ and $y$ in a group, let $[x,y]=xyx^{-1}y^{-1}$.

\begin{prop}\label{3.1}
Let $f_W:X_W\to{}S^2$ be the genus-$g$ Lefschetz fibration with the
 monodromy $W=1$, where $W$ is as above, and let $a_1,b_1,\dots,a_g,b_g$
 be the generators of $\pi_1(\Sigma_g)$ as shown in
{\rm Figure~\ref{abc}}.
Then we have followings:
\begin{enumerate}
 \item (See \cite{k2}.)
       Let $U=WW^{t_{b_1}}\cdots{}W^{t_{b_g}}$, then the fundamental
       group $\pi_1(X_U)$ of the Lefschetz fibration $X_U$ has the
       following presentation
       $$
       \pi_1(X_U)=
       \left\{
       \begin{array}{ll}
	\left<
	a_1,b_1,\dots,a_g,b_g
	\left|
	\begin{array}{l}
	b_1,\dots,b_g,\\
	a_1a_g,\dots,a_{\frac{g}{2}}a_{\frac{g+2}{2}}
	\end{array}
	\right.
	\right>
	&\textrm{when}~g~\textrm{is even},\\
	\left<
	a_1,b_1,\dots,a_g,b_g
	\left|
	\begin{array}{l}
	b_1,\dots,b_g,\\
	a_1a_g,\dots,a_{\frac{g-1}{2}}a_{\frac{g+3}{2}},\\
	a_{\frac{g+1}{2}}
	\end{array}
	\right.
	\right>
	&\textrm{when}~g~\textrm{is odd},
       \end{array}
       \right.
       $$
       and, the group $\pi_1(X_U)$ is isomorphic to the free group of rank
       $[\frac{g}{2}]$.
 \item Let $U'=WW^{t_{b_2}}\cdots{}W^{t_{b_{g-1}}}$, then the fundamental
       group $\pi_1(X_{U'})$ of the Lefschetz fibration $X_{U'}$ has the
       following presentation
       $$
       \pi_1(X_{U'})=
       \left\{
       \begin{array}{ll}
	\left<
	a_1,b_1,\dots,a_g,b_g
	\left|
	\begin{array}{l}
	[a_1,b_1],\\
	b_2,\dots,b_{g-1},\\
	b_1b_g,\\
	a_1a_g,\dots,a_{\frac{g}{2}}a_{\frac{g+2}{2}}
	\end{array}
	\right.
	\right>
	&\textrm{when}~g~\textrm{is even},\\
	\left<
	a_1,b_1,\dots,a_g,b_g
	\left|
	\begin{array}{ll}
	[a_1,b_1],\\
	b_2,\dots,b_{g-1},\\
	b_1b_g,\\
	a_1a_g,\dots,a_{\frac{g-1}{2}}a_{\frac{g+3}{2}},\\
	a_{\frac{g+1}{2}}
	\end{array}
	\right.
	\right>
	&\textrm{when}~g~\textrm{is odd},
       \end{array}
       \right.
       $$
       and, the group $\pi_1(X_{U'})$ is isomorphic to the free product
       of the free group of rank $([\frac{g}{2}]-1)$ with $\Z\oplus\Z$.
\end{enumerate}
\end{prop}

\begin{proof}
Simple closed curves $B_0,\dots,B_g$ and $a,b,c$ as shown in
 Figure~\ref{mkc} can be described in $\pi_1(\Sigma_g)$, up to
 conjugation, as follows
\begin{itemize}
 \item $B_{2k}=a_kb_{k+1}b_{k+2}\cdots{}b_{g-k-1}b_{g-k}c_{g-k}a_{g-k+1}$, 
       where $0\leq{}k\leq\frac{g}{2}$,
 \item $B_{2k+1}=a_{k+1}b_{k+1}b_{k+2}\cdots{}b_{g-k-1}b_{g-k}c_{g-k}a_{g-k}$,
       where $0\leq{}k\leq\frac{g}{2}$,
 \item $a=a_{\frac{g+1}{2}}$, $b=c_{\frac{g-1}{2}}a_{\frac{g+1}{2}}$ and
       $c=c_{\frac{g}{2}}$,
\end{itemize} 
where let $a_0=a_{g+1}=1$.
In addition, note that
$c_i=b_i^{-1}\cdots{}b_1^{-1}(a_1b_1a_1^{-1})\cdots(a_ib_ia_i^{-1})$ up
 to conjugation, for $1\leq{}i\leq{}g$. 
Since $X_W$ has a section, by Proposition~\ref{2.1}, we first obtain a
 presentation of $\pi_1(X_W)$ as follows.
$$
\pi_1(X_W)=\left\{
\begin{array}{ll}
\left\la
a_1,b_1,\dots,a_g,b_g
\left|
\begin{array}{l}
c_g,c_{\frac{g}{2}},\\
a_1a_g,\dots,a_{\frac{g}{2}}a_{\frac{g+2}{2}},\\
b_1a_gb_ga_g^{-1},\dots,b_{\frac{g}{2}}a_{\frac{g+2}{2}}b_{\frac{g+2}{2}}a_{\frac{g+2}{2}}^{-1}
\end{array}
\right.
\right\ra&\textrm{when}~g~\textrm{is even},\\
\left\la
a_1,b_1,\dots,a_g,b_g
\left|
\begin{array}{l}
c_g,a_{\frac{g+1}{2}},b_{\frac{g+1}{2}},c_{\frac{g-1}{2}},\\
a_1a_g,\dots,a_{\frac{g-1}{2}}a_{\frac{g+3}{2}},\\
b_1a_gb_ga_g^{-1},\dots,b_{\frac{g-1}{2}}a_{\frac{g+3}{2}}b_{\frac{g+3}{2}}a_{\frac{g+3}{2}}^{-1}
\end{array}
\right.
\right\ra&\textrm{when}~g~\textrm{is odd}.
\end{array}
\right.
$$
(We have that $\pi_1(X_W)$ is isomorphic to $\pi_1(\Sigma_{[\frac{g}{2}]})$.)
Since each $b_i$ intersects some $B_j$ transversely at only one point,
 by Proposition~\ref{2.2}, we obtain the claim.
\end{proof}

{\bf Remark.}
From Proposition~\ref{3.1}, we have followings.
\begin{itemize}
 \item For $n\geq1$, there are genus-$2n$ and $(2n+1)$ Lefschetz
       fibrations whose fundamental groups are isomorphic to the free
       group of rank $n$.
 \item For $n\geq2$, there are genus-$(2n-2)$ and $(2n-1)$ Lefschetz
       fibrations whose fundamental groups are isomorphic to the free
       product of the free group of rank $(n-2)$ with $\Z\oplus\Z$.
\end{itemize}

\begin{figure}[h]
\subfigure[]{\includegraphics[scale=1.0]{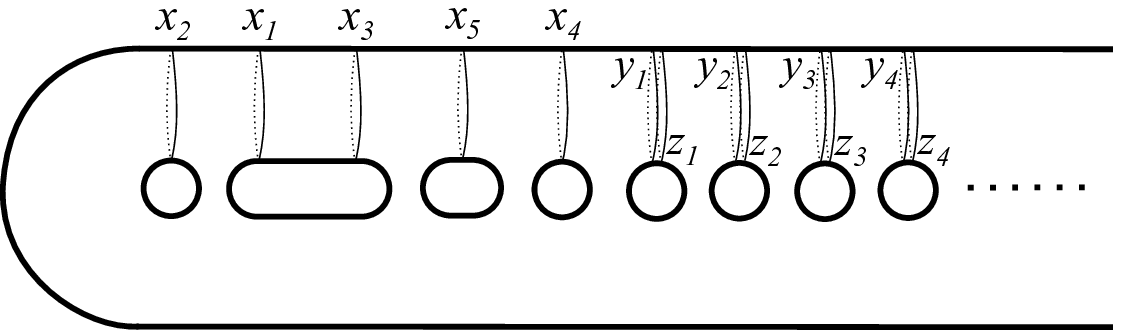}}

\subfigure[]{\includegraphics[scale=1.0]{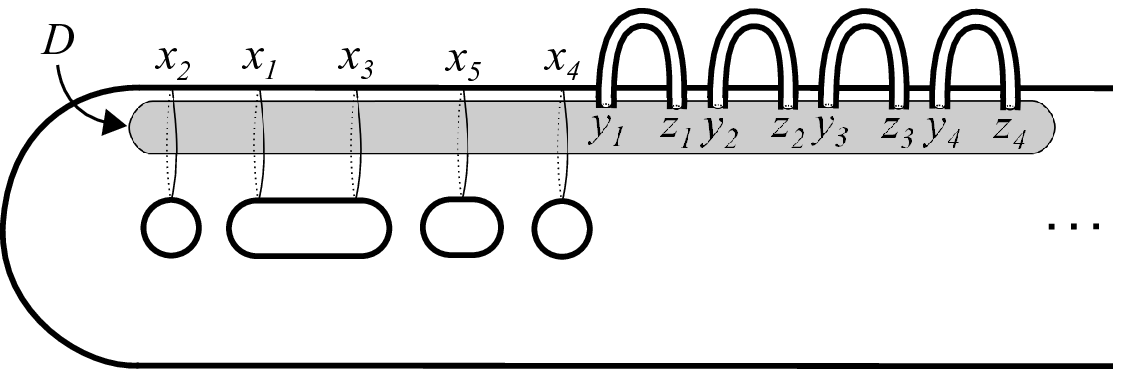}}

\subfigure[]{\includegraphics[scale=1.0]{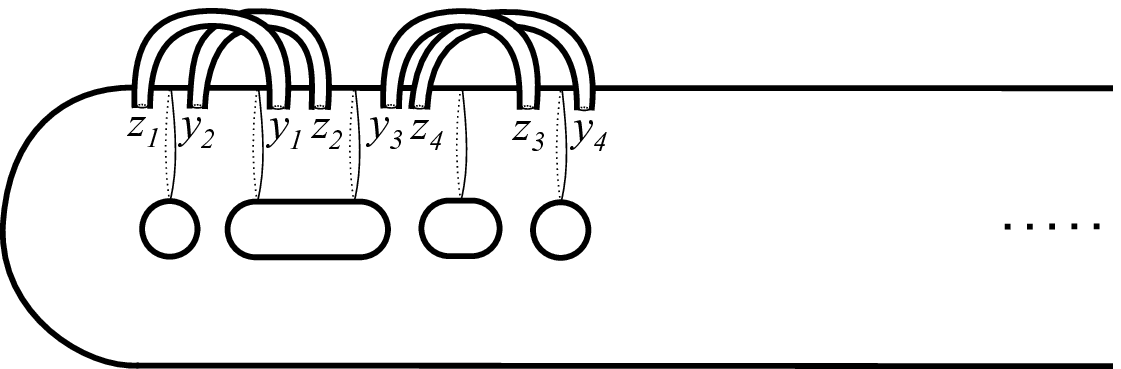}}
\caption{The loop $R$ in the case $n=4,r=g_2g_1g_2^2g_4^{-1}g_3^{-2}$.}\label{r1}
\end{figure}

\begin{figure}[h]
\subfigure[]{\includegraphics[scale=1.0]{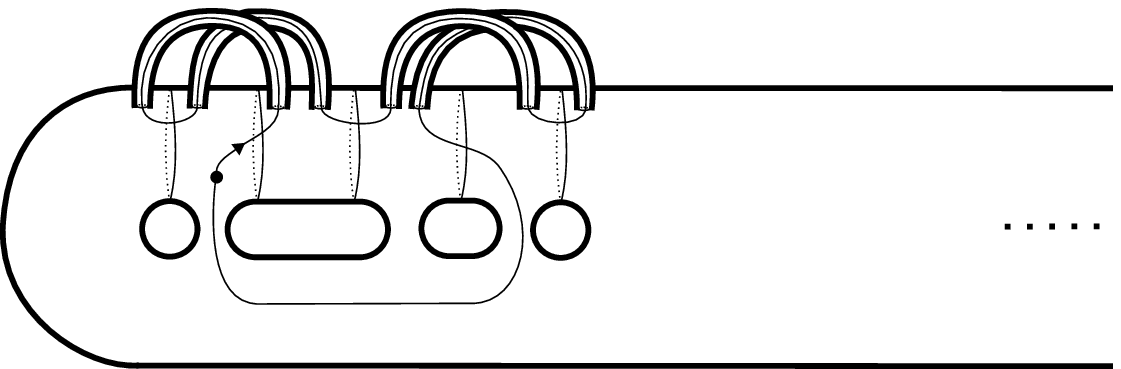}}

\subfigure[The loop $R.$]{\includegraphics[scale=1.0]{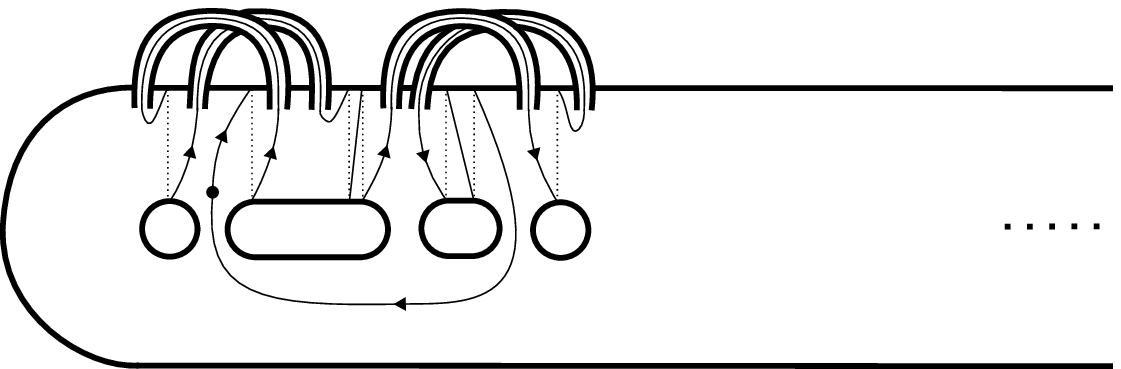}}

\subfigure[]{\includegraphics[scale=1.0]{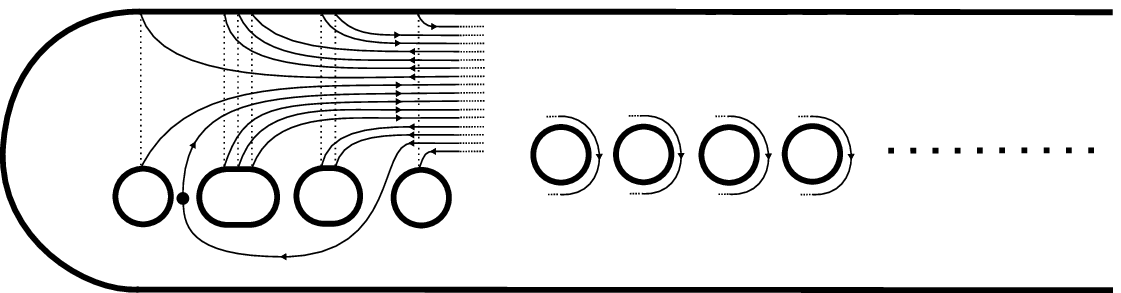}}
\caption{The loop $R$ in the case $n=4,r=g_2g_1g_2^2g_4^{-1}g_3^{-2}$.}\label{r2}
\end{figure}

\begin{figure}[h]
\includegraphics[scale=0.6]{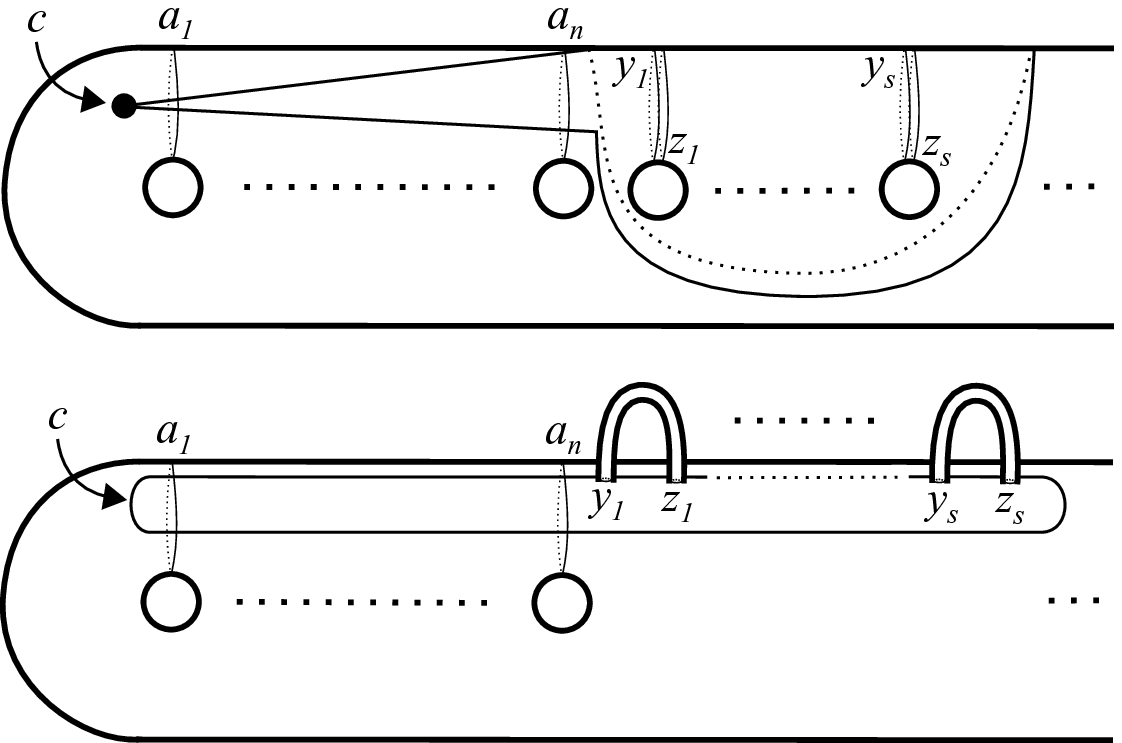}
\caption{The loop $c$ where $s=\ell(r_i)-1$.}\label{c}
\end{figure}

Let $\Gamma$ be a finitely presented group with a presentation
$\Gamma=\la{}g_1,\dots,g_n\mid{}r_1,\dots,r_k\ra$ and let
$\displaystyle\ell=\max_{1\leq{}i\leq{}k}\{\ell(r_i)\}$.
For $g\geq{}n+\ell-1$ and $r_i$, we construct a simple closed curve $R_i$
on $\Sigma_g$ as below.

At first, we construct a simple closed curve $R$ in the case $n=4$ and 
$r=g_2g_1g_2^2g_4^{-1}g_3^{-2}$ as an example.
Note that $\ell(r)=5$.
Let $x_1,x_2,x_3,x_4,x_5$ be loops on $\Sigma_g$ which are homotopic to
$a_2,a_1,a_2,a_4$ and $a_3$, respectively, as shown in
Figure~\ref{r1}~(a).
Let $y_1,y_2,y_3,y_4$ be loops on $\Sigma_g$ which are homotopic to 
$a_5,a_6,a_7,a_8$, respectively, and let $z_1,z_2,z_3,z_4$ be loops on
$\Sigma_g$ which are homotopic to $a_5,a_6,a_7,a_8$, respectively, as
shown in Figure~\ref{r1}~(a).
First we deform $\Sigma_g$ around $y_1,z_1,\dots,y_4,z_4$ as shown in
Figure~\ref{r1}~(b).
Then let $D$ be a subsurfase containing $y_t$ and $z_t$ which is
surrounded by a simple closed curve on $\Sigma_g$ as shown in
Figure~\ref{r1}~(b).
Next, for $1\leq{}t\leq4$, we move $y_t$ to the right side of $x_t$ in
$D$, and $z_t$ to the left side of $x_{t+1}$ in $D$, as shown in
Figure~\ref{r1}~(c).
Let $\overline{R}$ be the loop as shown in Figure~\ref{r2}~(a), and let 
$R=(\overline{R})t_{x_1}^{-1}t_{x_2}^{-1}t_{x_3}^{-2}t_{x_4}t_{x_5}^2$,
as shown in Figure~\ref{r2}~(b).
Finally, we deform the surface so that $y_1,\dots,y_4$ and
$z_1,\dots,z_4$ go back to their original position as shown in
Figure~\ref{r2}~(c).

In general, a loop $R_i$ is constructed as follows.
Let $r_i=g_{j(1)}^{m(1)}\cdots{}g_{j(\ell(r_i))}^{m(\ell(r_i))}$.
For $1\leq{}t\leq\ell(r_i)$, let $x_t$ be a loop on $\Sigma_g$ which is
homotopic to $a_{j(t)}$.
If $j(s)=j(s')$ for some $s<s'$, we put $x_{s'}$ to the right side of
$x_s$.
For $1\leq{}t\leq\ell(r_i)-1$, let $y_t$ and $z_t$ be loops on $\Sigma_g$
which are homotopic to $a_{n+t}$, such that $z_t$ is in the right side
of $y_t$.

First we deform $\Sigma_g$ around
$y_1,z_1,\dots,y_{\ell(r_i)-1},z_{\ell(r_i)-1}$, similarly to the above
example.
Let $c$ be a simple closed curve which is described in
$\pi_1(\Sigma_g)$ as follows
$$c=(a_{n+1}b_{n+1}a_{n+1}^{-1})\cdots(a_{n+\ell(r_i)-1}b_{n+\ell(r_i)-1}a_{n+\ell(r_i)-1}^{-1})b_{n+\ell(r_i)-1}^{-1}\cdots{}b_{n+1}^{-1},$$
and intersects each of $a_1,\dots,a_n$ at two points, as shown in
Figure~\ref{c}.
Then let $D$ be a subsurface whose boundary is $c$, and which contains
$y_t$ and $z_t$.

Next, for $1\leq{}t\leq\ell(r_i)-1$, we move $y_t$ to the right side of
$x_t$ in $D$, and $z_t$ to the left side of $x_{t+1}$ in $D$.
We regard that this motion does not affect on loops $a_i,b_i$ and $c_i$.
Hence $x_1,\dots,x_{\ell(r_i)}$ also do not deform, as shown in
Figure~\ref{r1}~(a).

After that, we define a simple closed curve as shown in
Figure~\ref{r2}~(a).
More precisely, we construct arcs $L_i$ and $L_i'$ as follows.
The arc $L_i$ is in $D$.
$L_i$ begins from the point at the left side of $x_1$ on the loop $c$,
crosses $x_1,y_1,z_1,x_2,y_2,z_3,\dots$, in this order, finally crosses
$x_{\ell(r_i)}$, and stops at the right side of $x_{\ell(r_i)}$ on the
loop $c$.
Let $L_i'$ be an arc whose base point is the end point of $L_i$, end
point is the base point of $L_i$, and which does not intersect the
interior of $D$ and loops $a_1,b_1,\dots,a_n,b_n$ and $c_n$.
Note that the surface which is obtained by removing loops $c$,
$a_1,b_1,\dots,a_n,b_n$ and $c_n$ from $\Sigma_g$, and which contains
$L_i'$ is a disk.
Hence the arc $L_i'$ is unique up to homotopy relative to the base
point and the end point.
Let $L_i\cdot{}L_i'$ denote the composition of $L_i$ and $L_i'$.

We now define 
$R_i=(L_i\cdot{}L_i')t_{x_1}^{-m(1)}\cdots{}t_{x_{\ell(r_i)}}^{-m(\ell(r_i))}$.
Finally, we deform the surface so that
$y_1,z_1,\dots,y_{\ell(r_i)-1},z_{\ell(r_i)-1}$ go back to their
original position.

Note that the loop $R_i$ is described in $\pi_1(\Sigma_g)$, up to
conjugation, as the following 
\begin{eqnarray}
R_i&=&(\prod_{1\leq{}t\leq{}m(1)}x_{i,1,t}a_{j(1)})\cdots(\prod_{1\leq{}t\leq{}m(\ell(r_i))}x_{i,\ell(r_i),t}a_{j(\ell(r_i))})\widetilde{L_i},
\end{eqnarray}
where $x_{i,s,t}$ is a loop which is some products of
$a_{n+1},b_{n+1},\dots,a_{\ell(r_i)-1},b_{\ell(r_i)-1}$ and $c_{n+1}$,
and $\widetilde{L_i}$ is a loop which is described in $\pi_1(\Sigma_g)$
as the following
$$
\widetilde{L_i}=
\left\{
\begin{array}{ll}
b_{j(\ell(r_i))}^{-1}b_{j(\ell(r_i))-1}^{-1}\cdots{}b_{j(1)+1}^{-1}b_{j(1)}^{-1}&\textrm{when}~j(1)\leq{}j(\ell(r_i)),\\
b_{j(\ell(r_i))+1}b_{j(\ell(r_i))}\cdots{}b_{j(1)}b_{j(1)-1}&\textrm{when}~j(1)>j(\ell(r_i)).
\end{array}
\right.
$$

We now prove Theorem~\ref{1.1}.

\begin{proof}[Proof of Theorem 1.1]
For $g\geq2n+\ell-1$, let $V$ be the following
$$V=UW^{t_{a_{n+1}}}\cdots{}W^{t_{a_{[\frac{g}{2}]}}},$$
where $U=WW^{t_{b_1}}\cdots{}W^{t_{b_g}}$.
In addition, let $V'$ be the following
$$V'=VV^{t_{R_1}}\cdots{}V^{t_{R_k}},$$
where $R_i$ is the loop constructed previously. 
We show that the fundamental group $\pi_1(X_{V'})$ is isomorphic to
 $\Gamma$.

Since each of $b_1,\dots,b_g$ and $a_{n+1},\dots,a_{[\frac{g}{2}]}$
 intersects some $B_i$ transversely at only one point, by
 Proposition~\ref{2.2}, we have
\begin{eqnarray*}
\pi_1(X_V)
&=&
\pi_1(\Sigma_g)/{\la{}b_1,\dots,b_g,a_{n+1},\dots,a_{[\frac{g}{2}]}\ra}\\
&=&
\pi_1(X_{U})/{\la{}a_{n+1},\ldots,a_{[\frac{g}{2}]}\ra}.
\end{eqnarray*}
In addition, by the presentation of (1) of Proposition \ref{3.1}, we
 have 
$$\pi_1(U)=\la{}a_1,\dots,a_{[\frac{g}{2}]}\ra.$$
Therefore we have 
\begin{eqnarray*}
\pi_1(X_V)
&=&
\la{}a_1,\dots,a_{[\frac{g}{2}]}\mid{}a_{n+1},\dots,a_{[\frac{g}{2}]}\ra\\
&=&
\la{}a_1,\dots,a_n\ra,
\end{eqnarray*}
Because of the presentation of $\pi_1(X_U)$ in (1) of
 Proposition~\ref{3.1}, we assume $g\geq2n+\ell-1$ in place of
 $g\geq{}n+\ell-1$.

For any $1\leq{}i\leq{}k$, consider the vanishing cycle
 $((B_0)t_{a_{n+1}})t_{R_i}$ of $X_{V'}$.
Note that $(B_0)t_{a_{n+1}}$ and $(a_{n+1})t_{R_i}$ are described in
 $\pi_1(\Sigma_g)$ as followings
\begin{itemize}
 \item $(B_0)t_{a_{n+1}}=a_{n+1}(b_1\cdots{}b_g)$,
 \item $(a_{n+1})t_{R_i}=a_{n+1}(zR_iz^{-1})$ for some
       $z\in\pi_1(\Sigma_g)$.
\end{itemize}
Then, we have that $((B_0)t_{a_{n+1}})t_{R_i}$ is described in 
$\pi_1(\Sigma_g)$ as the following 
\begin{eqnarray*}
((B_0)t_{a_{n+1}})t_{R_i}
&=&
(x\cdot{}a_{n+1}(b_1\cdots{}b_n)\cdot{}x^{-1})t_{R_i}\\
&=&
(x)t_{R_i}(a_{n+1})t_{R_i}(b_1\cdots{}b_n)t_{R_i}(x^{-1})t_{R_i}\\
&=&
(x)t_{R_i}(y\cdot{}a_{n+1}(zR_iz^{-1})\cdot{}y^{-1})(w\cdot(B_0)t_{R_i}\cdot{}w^{-1})((x)t_{R_i})^{-1},
\end{eqnarray*}
for some elements $x,y$ and $w$ in $\pi_1(\Sigma_g)$.
Since $a_{n+1}=(B_0)t_{R_i}=1$ in $\pi_1(X_{V'})$, we have $R_i=1$ from
 $((B_0)t_{a_{n+1}})t_{R_i}=1$, in $\pi_1(X_{V'})$.
For a vanishing cycle $c$ of $X_V$, if $R_i$ intersects $c$ transversely
 at $s$ points, then the vanishing cycle $(c)t_{R_i}$ of $X_{V'}$ is
 described in $\pi_1(\Sigma_g)$, up to conjugation, as the following
$$(c)t_{R_i}=x_1R_i^{\varepsilon_1}\cdots{}x_sR_i^{\varepsilon_s}x_{s+1},$$ 
where $\varepsilon_j=\pm1$ and  $x_1,\dots,x_{s+1}$ are elements in
$\pi_1(\Sigma_g)$ such that $c=x_1\cdots{}x_{s+1}$.
Since $R_i=1$ and $c=1$ in $\pi_1(X_{V'})$, we can delete the relation
 $(c)t_{R_i}=1$ of $\pi_1(X_{V'})$.
We now define
 $\hat{r}_i=a_{j(1)}^{m(1)}\cdots{}a_{j(\ell(r_i))}^{m(\ell(r_i))}$ for
 $r_i=g_{j(1)}^{m(1)}\cdots{}g_{j(\ell(r_i))}^{m(\ell(r_i))}$.
Since $x_{i,s,t}$ and $\widetilde{L_i}$ in (1) is $1$ in $\pi_1(X_{V'})$,
the natural epimorphism $\pi_1(\Sigma_g)\twoheadrightarrow\pi_1(X_{V'})$
 sends $R_i$ to $\hat{r}_i$.
Note that the vanishing cycles of $X_{V'}$ consist of $c$ and
 $(c)t_{R_i}$ for all vanishing cycles $c$ of $X_V$ and $1\leq{}i\leq{}k$.
Therefore, we have 
\begin{eqnarray*}
\pi_1(X_{V'})
&=&
\la{}a_1,\dots,a_n\mid\hat{r}_1,\dots,\hat{r}_k\ra\\
&\cong&
\Gamma.
\end{eqnarray*}

Thus, the proof of Theorem 1.1 is completed.
\end{proof}

\section{Proof of Theorem~\ref{1.2}}

In this section, we prove Theorem 1.2.

\subsection{Proof of (1) of Theorem~\ref{1.2}}

\begin{figure}[h]
\subfigure[The loop $R_{1,k}$ with $k=2$.]{\includegraphics{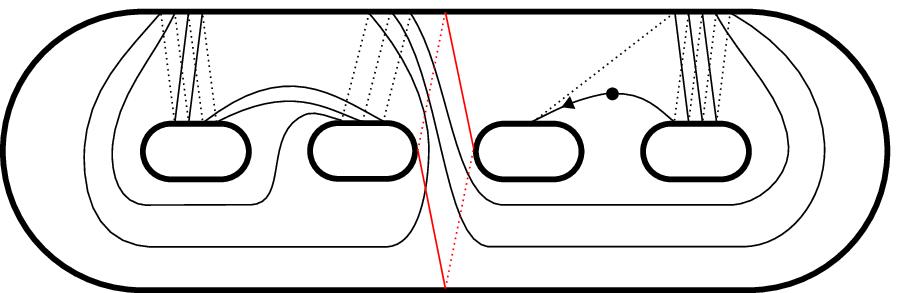}}

\subfigure[The loop $R_2$.]{\includegraphics{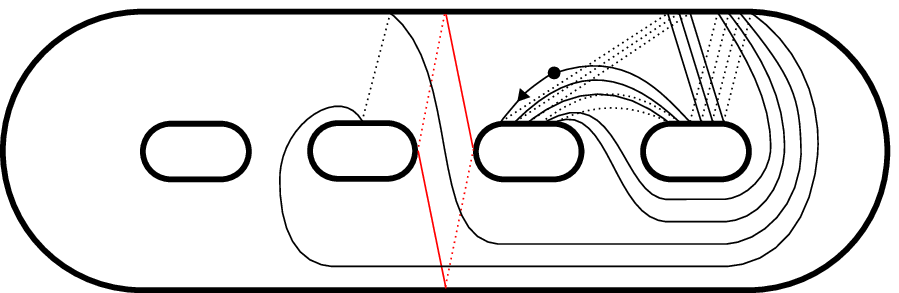}}

\subfigure[The loop $R_3$ with $n=4$.]{\includegraphics{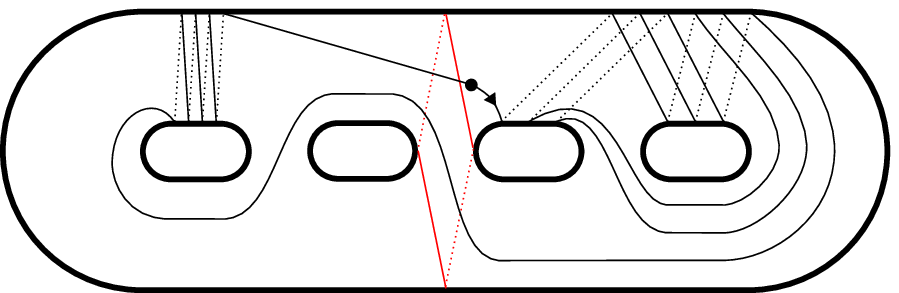}}
\caption{}\label{(1)}
\end{figure}

For $n\geq 2$, let $B_n$ denote the $n$-strands braid group. 
The group $B_n$ has a presentation with generators
$\sigma_1,\dots,\sigma_{n-1}$ and with relations
\begin{itemize}
 \item $\sigma_i\sigma_j\sigma_i^{-1}\sigma_j^{-1}=1$, where
       $1\leq{}i<j-1\leq{}n-2$,
 \item $\sigma_i\sigma_{i+1}\sigma_i\sigma_{i+1}^{-1}\sigma_i^{-1}\sigma_{i+1}^{-1}=1$,
       where $1\leq{}i\leq{}n-2$.
\end{itemize}
Let $x=\sigma_1$ and $y=\sigma_1\cdots\sigma_{n-1}$.
Then $B_n$ can be presented with generators $x,y$ and with relations
\begin{itemize}
 \item $xy^kxy^{-k}x^{-1}y^kx^{-1}y^{-k}=1$, where $2\leq{}k\leq{}n-2$,
 \item $xyxy^{-1}xyx^{-1}y^{-1}x^{-1}yx^{-1}y^{-1}=1$,
 \item $(xy)^{n-1}y^{-n}=1$.
\end{itemize}
A correspondence between the first presentation and the second
presentation is given by $\sigma_i=y^{i-1}xy^{1-i}$ for
$1\leq{}i\leq{}n-1$. 
See \cite{k2} for this presentation.

We now prove (1) of Theorem~\ref{1.2}.

\begin{proof}[Proof of (1) of Theorem~\ref{1.2}]
Since $B_2$ is isomorphic to $\Z$, we have $g(B_2)=1$ from
 Theorem~\ref{5.1} (cf. \cite{k2}).
For $n\geq3$, since $B_n$ is generated by two generators $x,y$, we have
 $g(B_n)\geq2$ from (2) of Theorem~\ref{2.3} (cf. \cite{k2}).
Therefore, we prove $g(B_n)\leq4$ for $n\geq3$.

Let $R_{1,k},R_2$ and $R_3$ be simple closed curves on $\Sigma_4$ as
 shown in Figure~\ref{(1)}, where $2\leq{}k\leq{}n-2$.
Note that $R_{1,k},R_2$ and $R_3$ intersect $B_4$ transversely at only
 one point, for $2\leq{}k\leq{}n-2$.
Loops $R_{1,k},R_2$ and $R_3$ can be described in $\pi_1(\Sigma_4)$, up
 to conjugation, as follows
\begin{itemize}
 \item $R_{1,k}=\3^{-1}\4^{-k}(\8\9)^{-1}\2\1^{-k}(\6)\2^{-1}(\6\7)^{-1}\1^{k}\2^{-1}(\8\9)\4^k$,
       where $2\leq{}k\leq{}n-2$,
 \item $R_2=\3^{-1}\4^{-1}(\9^{-1})\3^{-1}\4\3^{-1}\4^{-1}(\7\8\9)^{-1}\2^{-1}(\8\9)\4\3\4^{-1}\3(\9)\4$, 
 \item $R_3=(\3^{-1}\4^{-1}(\9^{-1}))^{n-1}(\6\8)^{-1}\1^{-n}$.
\end{itemize}
Let $V_1$ be the following:
$$
V_1=WW^{t_{\6}}W^{t_{\7}}W^{t_{\8}}W^{t_{\9}}(\prod_{2\leq{}k\leq{}n-2}W^{t_{R_{1,k}}})W^{t_{R_2}}W^{t_{R_3}}.
$$
Then, from Proposition~\ref{2.2} and (1) of Proposition~\ref{3.1}, the
 fundamental group $\pi_1(X_{V_1})$ can be presented with generators
 $\2,\1$ and with relations
\begin{itemize}
 \item $\2\1^k\2\1^{-k}\2^{-1}\1^k\2^{-1}\1^{-k}=1$, where 
       $2\leq{}k\leq{}n-2$,
 \item $\2\1\2\1^{-1}\2\1\2^{-1}\1^{-1}\2^{-1}\1\2^{-1}\1^{-1}=1$,
 \item $(\2\1)^{n-1}\1^{-n}=1$.
\end{itemize}
Let $\2=x$ and $\1=y$.
Then it follows that $\pi_1(X_{V_1})$ is isomorphic to $B_n$. 
Therefore, for $n\geq3$ we have $g(B_n)\leq4$. 

Thus, the proof of (1) of Theorem~\ref{1.2} is completed.
\end{proof}

\subsection{Proof of (2) of Theorem~\ref{1.2}}

\begin{figure}[h]
\subfigure[The loop $R_4$.]{\includegraphics{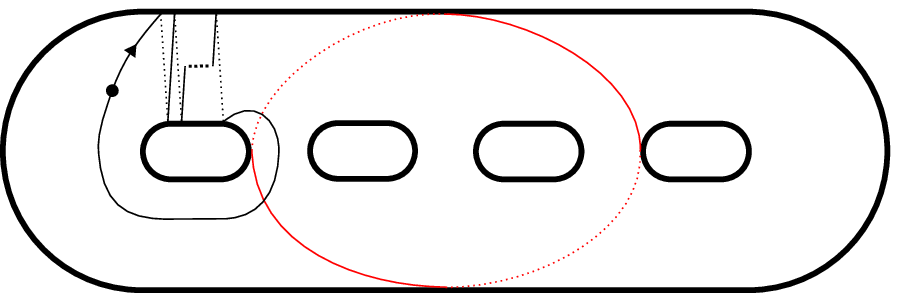}}

\subfigure[The loop $R_5$.]{\includegraphics{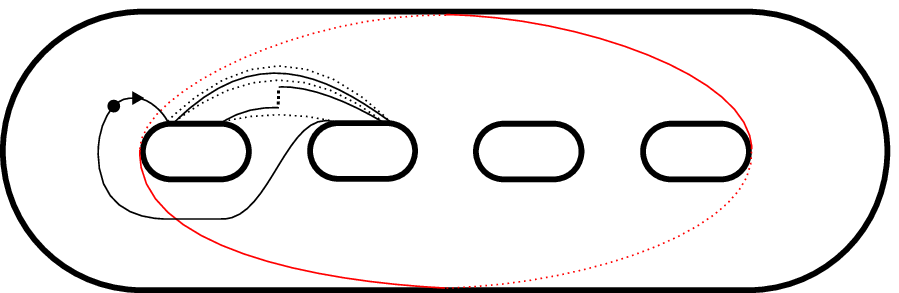}}

\subfigure[The loop $R_6$.]{\includegraphics{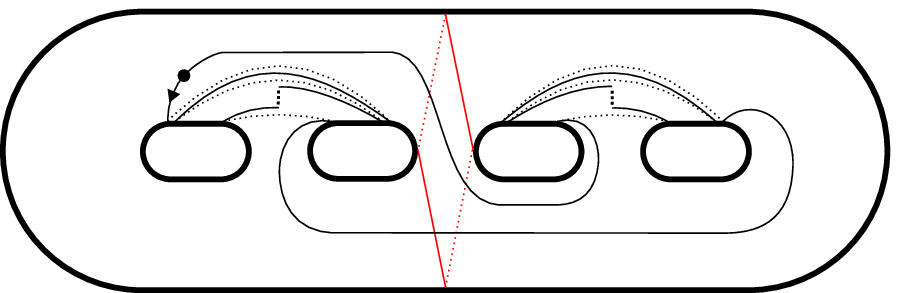}}
\caption{}\label{(2)}
\end{figure}

For $g\geq1$, let $\Hy_g$ be the hyperelliptic mapping class group of 
$\Sigma_g$, that is, a subgroup of the mapping class group $\M_g$ which
consists of elements commutative with a hyperelliptic involution.
It is well known that there is the natural epimorphism
$B_{2g+2}\twoheadrightarrow\Hy_g$. 
For $g\geq2$, Birman and Hilden \cite{bh} gave a presentation of the
group $\Hy_g$ with generators $\sigma_1,\dots,\sigma_{2g+1}$ and with
relations
\begin{itemize}
 \item $\sigma_i\sigma_j\sigma_i^{-1}\sigma_j^{-1}=1$, where
       $1\leq{}i<j-1\leq2g$,
 \item $\sigma_i\sigma_{i+1}\sigma_i\sigma_{i+1}^{-1}\sigma_i^{-1}\sigma_{i+1}^{-1}=1$,
       where $1\leq{}i\leq2g$,
 \item $(\sigma_1\cdots\sigma_{2g+1})^{2g+2}=1$,
 \item $(\sigma_1\cdots\sigma_{2g+1}\sigma_{2g+1}\cdots\sigma_1)^2=1$,
 \item $[\sigma_1\cdots\sigma_{2g+1}\sigma_{2g+1}\cdots\sigma_1,\sigma_1]=1$.
\end{itemize}

Similarly to Subsection~4.1, let $x=\sigma_1$ and
$y=\sigma_1\cdots\sigma_{2g+1}$.
Then, note that $y^{2g+2}=1$.
We calculate 
\begin{eqnarray*}
\sigma_1\cdots\sigma_{2g+1}\sigma_{2g+1}\cdots\sigma_1
&=&
y(y^{2g}xy^{-2g})\cdots(yxy^{-1})x\\
&=&
y^{2g+1}(xy^{-1})^{2g}x\\
&=&
y^{-1}(xy^{-1})^{2g}x\\
&=&
(y^{-1}x)^{2g+1}.
\end{eqnarray*}
Then we have
$(\sigma_1\cdots\sigma_{2g+1}\sigma_{2g+1}\cdots\sigma_1)^2=(y^{-1}x)^{4g+2}$.
In addition, we have
\begin{eqnarray*}
[\sigma_1\cdots\sigma_{2g+1}\sigma_{2g+1}\cdots\sigma_1,\sigma_1]
&=&
(y^{-1}x)^{2g+1}x(x^{-1}y)^{2g+1}x^{-1}\\
&=&
(y^{-1}x)^{2g+1}(yx^{-1})^{2g+1}.
\end{eqnarray*}
Therefore, $\Hy_g$ can be presented with generators $x,y$ and with
relations
\begin{itemize}
 \item $xy^kxy^{-k}x^{-1}y^kx^{-1}y^{-k}=1$, where $2\leq{}k\leq2g$,
 \item $xyxy^{-1}xyx^{-1}y^{-1}x^{-1}yx^{-1}y^{-1}=1$,
 \item $(xy)^{2g+1}y^{-2g-2}=1$,
 \item $y^{2g+2}=1$,
 \item $(y^{-1}x)^{4g+2}=1$,
 \item $(y^{-1}x)^{2g+1}(yx^{-1})^{2g+1}=1$.
\end{itemize}

We now prove (2) of Theorem~\ref{1.2}.

\begin{proof}[Proof of (2) of Theorem~\ref{1.2}]
For $g\geq1$, since $\Hy_g$ is generated by two generators $x,y$, we
 have $g(\Hy_g)\geq2$ from (2) of Theorem~\ref{2.3} (cf. \cite{k2}).
Therefore, we prove $g(\Hy_g)\leq4$ for $g\geq1$.

Let $R_4,R_5$ and $R_6$ be simple closed curves on $\Sigma_4$ as shown
 in Figure~\ref{(2)}.
Note that $R_4,R_5$ and $R_6$ intersect $B_2,B_1$  and $B_4$
 transversely at only one point, respectively.
Loops $R_4,R_5$ and $R_6$ can be described in $\pi_1(\Sigma_4)$, up to
 conjugation, as follows
\begin{itemize}
 \item $R_4=\1^{2g+2}(\6^{-1})$,
 \item $R_5=(\1^{-1}\2)^{4g+2}(\6^{-1})$,
 \item $R_6=(\1^{-1}\2)^{2g+1}(\7\8\9)(\4^{-1}\3)^{2g+1}(\8^{-1})$.
\end{itemize}
Let $V_2$ be the following:
$$
V_2=WW^{t_{\6}}W^{t_{\7}}W^{t_{\8}}W^{t_{\9}}(\prod_{2\leq{}k\leq2g}W^{t_{R_{1,k}}})W^{t_{R_2}}W^{t_{R_3}}W^{t_{R_4}}W^{t_{R_5}}W^{t_{R_6}}.
$$ 
Then, from Proposition~\ref{2.2} and (1) of Proposition~\ref{3.1}, the
 fundamental group $\pi_1(X_{V_2})$ can be presented with generators
 $\2,\1$ and with relations
\begin{itemize}
 \item $\2\1^k\2\1^{-k}\2^{-1}\1^k\2^{-1}\1^{-k}=1$, where
       $2\leq{}k\leq2g$,
 \item $\2\1\2\1^{-1}\2\1\2^{-1}\1^{-1}\2^{-1}\1\2^{-1}\1^{-1}=1$,
 \item $(\2\1)^{2g+1}\1^{-2g-2}=1$,
 \item $\1^{2g+2}=1$,
 \item $(\1^{-1}\2)^{4g+2}=1$,
 \item $(\1^{-1}\2)^{2g+1}(\1\2^{-1})^{2g+1}=1$.
\end{itemize}
Let $\2=x$ and $\1=y$.
Then it follows that $\pi_1(X_{V_2})$ is isomorphic to $\Hy_g$.
Therefore, for $g\geq2$ we have $g(\Hy_g)\leq4$.
In particular, since the group $\Hy_1$ is isomorphic to $\M_1$, we have
 $2\leq{}g(\Hy_1)\leq4$ from (3) of Theorem~\ref{2.3} (cf. \cite{k2}).

Thus, the proof of (2) of Theorem~\ref{1.2} is completed.
\end{proof}

\subsection{Proof of (3) of Theorem~\ref{1.2}}

\begin{figure}[h]
\subfigure[The loop $R_7$.]{\includegraphics{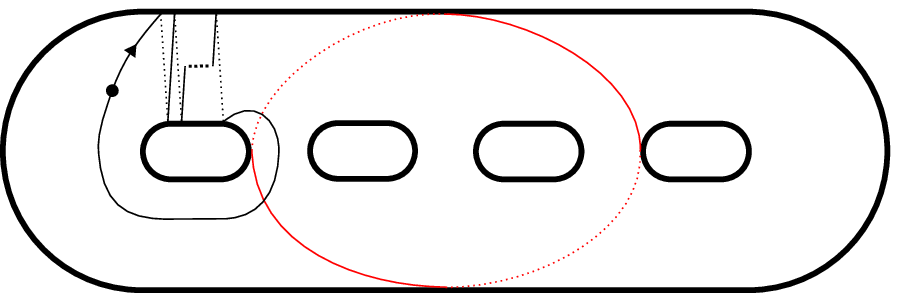}}

\subfigure[The loop $R_8$.]{\includegraphics{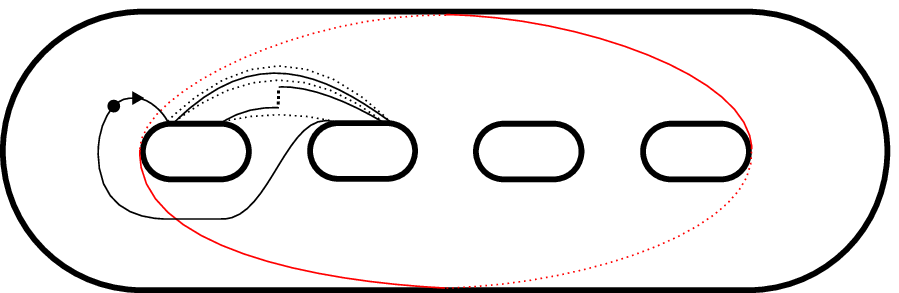}}
\caption{}\label{(3)}
\end{figure}

For $n\geq2$, let $\M_{0,n}$ denote the mapping class group of an
$n$-punctured sphere, that is, the group of isotopy classes of
orientation-preserving diffeomorphisms
$S^2\setminus\{p_1,\dots,p_n\}\to{}S^2\setminus\{p_1,\dots,p_n\}$.
Magnus \cite{m} gave a presentation of the group $\M_{0,n}$ with
generators $\sigma_1,\dots,\sigma_{n-1}$ and with relations
\begin{itemize}
 \item $\sigma_i\sigma_j\sigma_i^{-1}\sigma_j^{-1}=1$, where
       $1\leq{}i<j-1\leq{}n-2$,
 \item $\sigma_i\sigma_{i+1}\sigma_i\sigma_{i+1}^{-1}\sigma_i^{-1}\sigma_{i+1}^{-1}=1$,
       where $1\leq{}i\leq{}n-2$,
 \item $(\sigma_1\cdots\sigma_{n-1})^{n}=1$,
 \item $\sigma_1\cdots\sigma_{n-1}\sigma_{n-1}\cdots\sigma_1=1$.
\end{itemize}
Similarly to Subsection~4.1 and 4.2, let $x=\sigma_1$ and
$y=\sigma_1\cdots\sigma_{n-1}$.
Then $\M_{0,n}$ can be presented with generators $x,y$ and with relations
\begin{itemize}
 \item $xy^kxy^{-k}x^{-1}y^k x^{-1}y^{-k}=1$, where $2\leq{}k\leq{}n-2$,
 \item $xyxy^{-1}xyx^{-1}y^{-1}x^{-1}yx^{-1}y^{-1}=1$,
 \item $(xy)^{n-1}y^{-n}=1$,
 \item $y^n=1$,
 \item $(y^{-1}x)^{n-1}=1$.
\end{itemize}

We now prove (3) of Theorem~\ref{1.2}.

\begin{proof}[Proof of (3) of Theorem~\ref{1.2}]
Since the group $\M_{0,2}$ is isomorphic to $\Z_2$, we have
 $g(\M_{0,2})=2$ from Theorem~\ref{5.1} (cf. \cite{k2}).
For $n\geq3$, since $\M_{0,n}$ is generated by two generators $x,y$, we
 have $g(\M_{0,n})\geq2$ from (2) of Theorem~\ref{2.3} (cf. \cite{k2}).
Therefore, we prove $g(\M_{0,n})\leq4$ for $n\geq3$.

Let $R_7$ and $R_8$ be simple closed curves on $\Sigma_4$ as shown in
 Figure~\ref{(3)}.
Note that $R_7$ and $R_8$ intersect $B_2$ and $B_1$ transversely at only
 one point, respectively.
Loops $R_7$ and $R_8$ can be described in $\pi_1(\Sigma_4)$, up to
 conjugation, as follows
\begin{itemize}
 \item $R_7=\1^n(\6^{-1})$,
 \item $R_8=(\1^{-1}\2)^{n-1}(\6^{-1})$.
\end{itemize}
Let $V_3$ be the following:
$$V_3=V_1W^{t_{R_7}}W^{t_{R_8}}.$$
Then, from Proposition~\ref{2.2} and (1) of Proposition~\ref{3.1}, the
 fundamental group $\pi_1(X_{V_3})$ can be presented with generators
 $\2,\1$ and with relations
\begin{itemize}
 \item $\2\1^k\2\1^{-k}\2^{-1}\1^k\2^{-1}\1^{-k}=1$, where
       $2\leq{}k\leq{}n-2$,
 \item $\2\1\2\1^{-1}\2\1\2^{-1}\1^{-1}\2^{-1}\1\2^{-1}\1^{-1}=1$,
 \item $(\2\1)^{n-1}\1^{-n}=1$,
 \item $\1^n=1$,
 \item $(\1^{-1}\2)^{n-1}=1$. 
\end{itemize}
Let $\2=x$ and $\1=y$.
Then it follows that $\pi_1(X_{V_3})$ is isomorphic to $\M_{0,n}$.
Therefore, for $n\geq2$ we have $g(\M_{0,n})\leq4$. 

Thus, the proof of (3) of Theorem~\ref{1.2} is completed.
\end{proof}

\subsection{Proof of (4) of Theorem~\ref{1.2}}

\begin{figure}[h]
\includegraphics{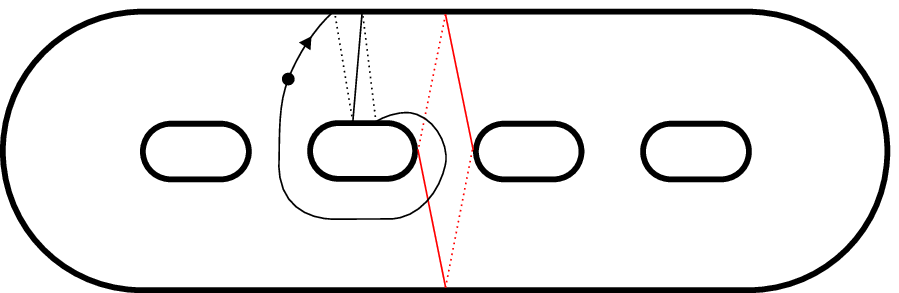}
\caption{The loop $R_9$.}\label{(4)}
\end{figure}

For $n\geq2$, let $S_n$ denote the $n$-symmetric group.
It is well known that the group $S_n$ has a presentation with generators
$\sigma_1,\dots,\sigma_{n-1}$ and with relations
\begin{itemize}
 \item $\sigma_i\sigma_j\sigma_i^{-1}\sigma_j^{-1}=1$, where
       $1\leq{}i<j-1\leq{}n-2$,
 \item $\sigma_i\sigma_{i+1}\sigma_i\sigma_{i+1}^{-1}\sigma_i^{-1}\sigma_{i+1}^{-1}=1$,
       where $1\leq{}i\leq{}n-2$,
 \item $\sigma_i^2=1$, where $1\leq{}i\leq{}n-1$.
\end{itemize}
Similarly to Subsection~4.1, let $x=\sigma_1$ and
$y=\sigma_1\cdots\sigma_{n-1}$.
Since $\sigma_i=y^{i-1}xy^{1-i}$, $\sigma_i^2=1$ if and only if $x^2=1$.
Therefore $S_n$ can be presented with generators $x,y$ and with relations
\begin{itemize}
 \item $xy^kxy^{-k}x^{-1}y^kx^{-1}y^{-k}=1$, where $2\leq{}k\leq{}n-2$,
 \item $xyxy^{-1}xyx^{-1}y^{-1}x^{-1}yx^{-1}y^{-1}=1$,
 \item $(xy)^{n-1}y^{-n}=1$,
 \item $x^2 =1$.
\end{itemize}

We now prove (4) of Theorem~\ref{1.2}.

\begin{proof}[Proof of (4) of Theorem~\ref{1.2}]
Since the group $S_2$ is isomorphic to $\Z_2$, we have $g(S_2)=2$ from
 Theorem~\ref{5.1} (cf. \cite{k2}).
For $n\geq3$, since $S_n$ is generated by two generators $x,y$, we have
 $g(S_n)\geq2$ from (2) of Theorem~\ref{2.3} (cf. \cite{k2}).
Therefore, we prove $g(S_n)\leq4$ for $n\geq3$.

Let $R_9$ be the simple closed curve on $\Sigma_4$ as shown in
 Figure~\ref{(4)}.
Note that $R_9$ intersects $B_4$ transversely at only one point.
The loop $R_9$ can be described in $\pi_1(\Sigma_4)$, up to conjugation,
 as follows
\begin{itemize}
 \item $R_9=\2^2(\7^{-1})$.
\end{itemize}
Let $V_4$ be the following:
$$V_4=V_1W^{t_{R_9}}.$$
Then, from Proposition~\ref{2.2} and (1) of Proposition~\ref{3.1}, the
 fundamental group $\pi_1(X_{V_4})$ can be presented with generators
 $\2,\1$ and with relations
\begin{itemize}
 \item $\2\1^k\2\1^{-k}\2^{-1}\1^k\2^{-1}\1^{-k}=1$, where
       $2\leq{}k\leq{}n-2$,
 \item $\2\1\2\1^{-1}\2\1\2^{-1}\1^{-1}\2^{-1}\1\2^{-1}\1^{-1}=1$,
 \item $(\2\1)^{n-1}\1^{-n}=1$,
 \item $\2^{2}=1$.
\end{itemize}
Let $\2=x$ and $\1=y$.
Then it follows that $\pi_1(X_{V_4})$ is isomorphic to $S_n$.
Therefore, for $n\geq2$ we have $g(S_n)\leq4$.

Thus, the proof of (4) of Theorem~\ref{1.2} is completed.
\end{proof}

\subsection{Proof of (5) of Theorem 1.2}

\begin{figure}
\subfigure[The loop $R_{1,k}$ with $k=2$.]{\includegraphics{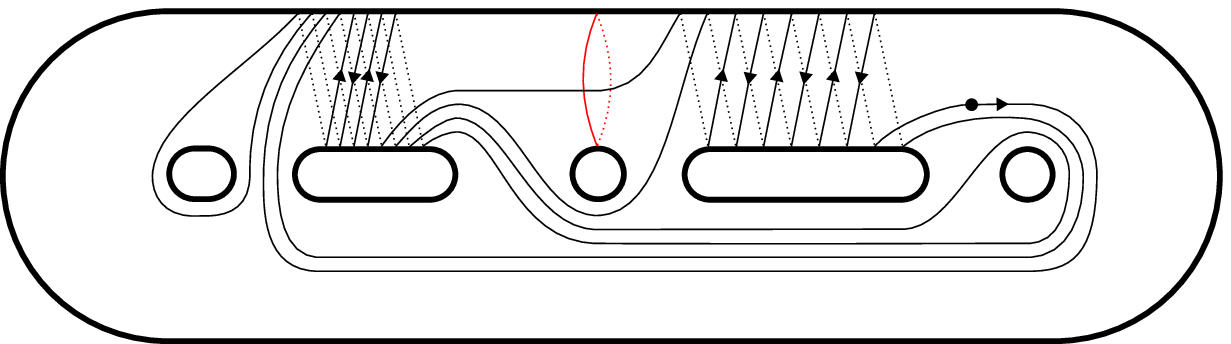}}

\subfigure[The loop $R_2$.]{\includegraphics{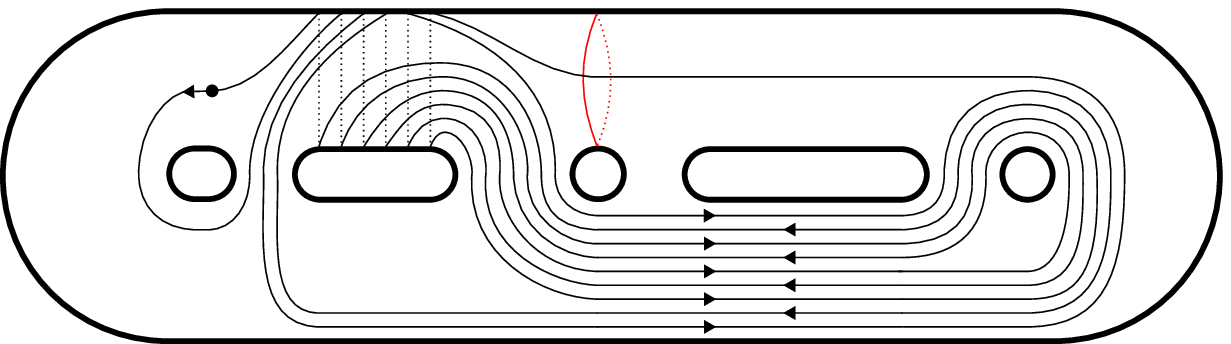}}

\subfigure[The loop $R_3$ with $n=3$.]{\includegraphics{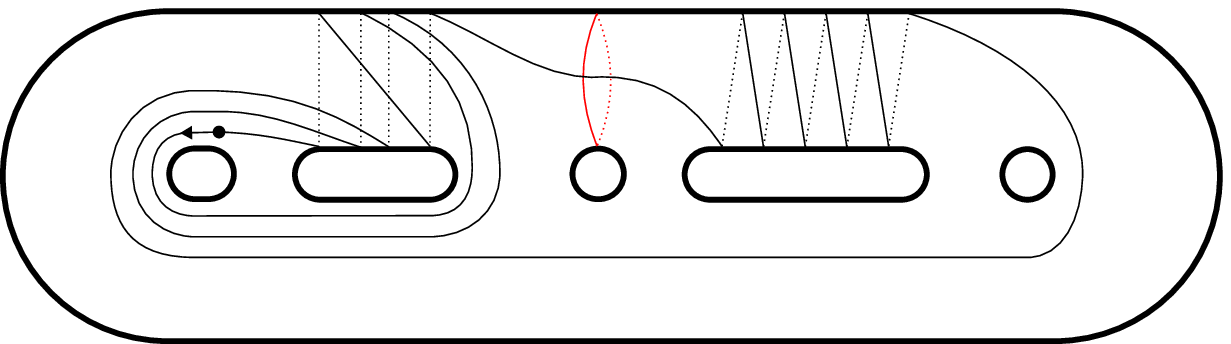}}

\subfigure[The loop $R_4$.]{\includegraphics{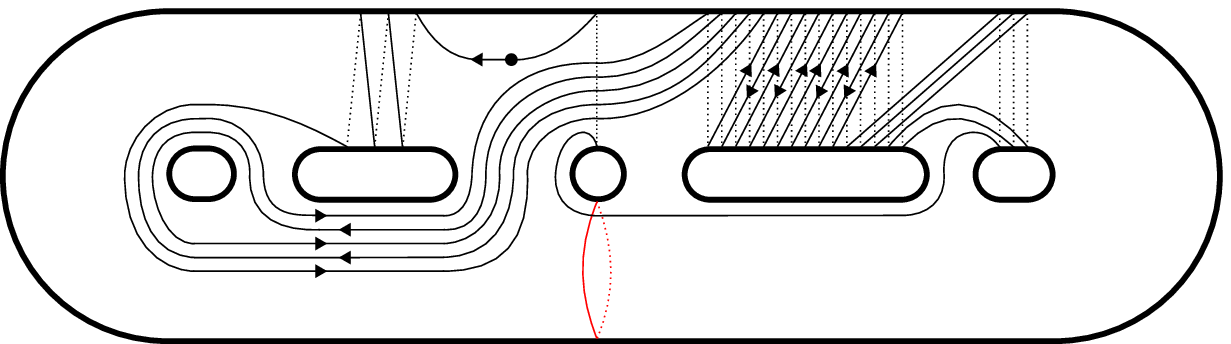}}

\subfigure[The loop $R_{5,i}$ with $i=3$.]{\includegraphics{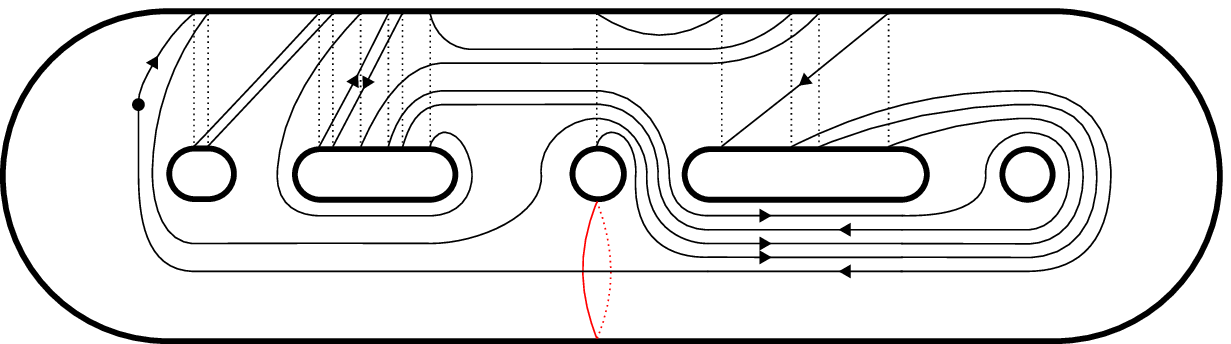}}
\caption{}\label{(5)}
\end{figure}

The Artin group is introduced by \cite{b}.
For $n\geq5$, the $n$-Artin group $\A_n$ associated to the Dynkin
diagram shown in Figure~\ref{artin} is defined by a presentation with
generators $\sigma_1,\dots,\sigma_{n-1},\tau$ and with relations
\begin{itemize}
 \item $\sigma_i\sigma_j\sigma_i^{-1}\sigma_j^{-1}=1$, where
       $1\leq{}i<j-1\leq{}n-2$,
 \item $\sigma_i\sigma_{i+1}\sigma_i\sigma_{i+1}^{-1}\sigma_i^{-1}\sigma_{i+1}^{-1}=1$,
       where $1\leq{}i\leq{}n-2$,
 \item $\sigma_4\tau\sigma_4\tau^{-1}\sigma_4^{-1}\tau^{-1}=1$,
 \item $\tau\sigma_i\tau^{-1}\sigma_i^{-1}=1$, where $1\leq{}i\leq{}n-1$
       with $i\neq4$.
\end{itemize}
It is known that there is the natural epimorphism
$\A_{2g+1}\twoheadrightarrow\M_g$.
Similarly to Subsection~4.1, let $x=\sigma_1$ and
$y=\sigma_1\cdots\sigma_{n-1}$.
In addition, let $z=\tau$.
Then the group $\A_n$ can be presented with generators $x,y,z$ and with
relations
\begin{itemize}
 \item $xy^kxy^{-k}x^{-1}y^kx^{-1}y^{-k}=1$, where $2\leq{}k\leq{}n-2$,
 \item $xyxy^{-1}xyx^{-1}y^{-1}x^{-1}yx^{-1}y^{-1}=1$,
 \item $(xy)^{n-1}y^{-n}=1$,
 \item $(y^3xy^{-3})z(y^3xy^{-3})z^{-1}(y^3x^{-1}y^{-3})z^{-1}=1$,
 \item $z(y^{i-1}xy^{1-i})z^{-1}(y^{i-1}x^{-1}y^{1-i})=1$, where
       $1\leq{}i\leq{}n-1$ with $i\neq4$.
\end{itemize}

We now prove (5) of Theorem~\ref{1.2}.

\begin{proof}[Proof of (5) of Theorem 1.2]
Since $\A _n $ is generated by three generators $x,y$ and $z$, we have
 $g(\A_n)\geq2$ from (2) of Theorem~\ref{2.3} (cf. \cite{k2}).
Therefore, we prove $g(\A_n)\leq5$.

Let $R_{1,k},R_2,R_3,R_4$ and $R_{5,i}$ be simple closed curves on
 $\Sigma_5$ as shown in Figure~\ref{(5)}, where $2\leq{}k\leq{}n-2$ and
 $2\leq{}i\leq{}n-1$ with $i\neq4$.
Note that we can not consider the loop $R_{5,1}$.
Note that $R_{1,k},R_2$ and $R_3$ intersect $a$ transversely at only one
 point, for $2\leq{}k\leq{}n-2$, and that $R_4$ and $R_{5,i}$ intersect
 $b$ transversely at only one point, for $2\leq{}i\leq{}n-1$ with
 $i\neq4$.
Loops $R_{1,k},R_2,R_3,R_4$ and $R_{5,i}$ can be described in
 $\pi_1(\Sigma_5)$, up to conjugation, as follows
\begin{itemize}
 \item $R_{1,k}=\0^{-1}(\7\8\9)^{-1}\2^k(\8\9)\0^{-1}(\8\9)^{-1}\2^{-k}(\7\8\9)\0\4^{-2k}(\8^{-1})\2^{-k}\6^{-1}\2^k\4^{2k}$,
       where $2\leq{}k\leq{}n-2$,
 \item $R_2=\6\2(\8\9)\0^{-1}(\8\9)^{-1}\2^{-1}(\8\9)\0^{-1}(\7\8\9)^{-1}\2(\8\9)\0(\8\9)^{-1}\\ 
       \hspace{0.9cm}
       \2^{-1}(\7\8\9)\0\2(\8\9)\0(\8\9)^{-1}\2^{-1}$,
 \item $R_3=(\6(\7)\2)^{n-1}(\6(\7\8\9)\0)\4^{n+2}\2^2$,
 \item $R_4=\2^3\6(\7)\4^3\5^{-1}\4^{-3}(\7^{-1})\6(\7)\4^3\5\4^{-3}(\7^{-1})\6^{-1}(\7)\4^3\5(\3\8\9)^{-1}$,
 \item $R_{5,i}=\1\2^{i-1}(\9)\0^{-1}(\9)\2^{1-i}\1^{-1}(\6(\7\9)\0)\4^{1-i}(\3\9)\0(\4^{2-i}\2^{2-i}(\7))\2^{-1}\4^{i-2}(\6(\7\8\9)\0)^{-1}$,
       where $2\leq{}i\leq{}n-1$ with $i\neq4$.
\end{itemize}
Let $V_5$ be the following:
$$
V_5=WW^{t_{\7}}W^{t_{\8}}W^{t_{\9}}(\prod_{2\leq{}k\leq{}n-2}W^{t_{R_{1,k}}})W^{t_{R_2}}W^{t_{R_3}}W^{t_{R_4}}(\prod_{2\leq{}i\leq{}n-1,i\neq4}W^{t_{R_{5,i}}}).
$$
Then, from Proposition~\ref{2.2} and (2) of Proposition~\ref{3.1}, the
 fundamental group $\pi_1(X_{V_5})$ can be presented with generators
 $\6,\2,\1$ and with relations
\begin{itemize}
 \item $\6\2^k\6\2^{-k}\6^{-1}\2^k\6^{-1}\2^{-k}=1$, where
       $2\leq{}k\leq{}n-2$,
 \item $\6\2\6\2^{-1}\6\2\6^{-1}\2^{-1}\6^{-1}\2\6^{-1}\2^{-1}=1$,
 \item $(\6\2)^{n-1}\2^{-n}=1$,
 \item $(\2^3\6\2^{-3})\1(\2^3\6\2^{-3})\1^{-1}(\2^3\6^{-1}\2^{-3})\1^{-1}=1$,
 \item $\1(\2^{i-1}\6\2^{1-i})\1^{-1}(\2^{i-1}\6^{-1}\2^{1-i})=1$, where
       $2\leq{}i\leq{}n-1$ with $i\neq4$,
 \item $\1\6\1^{-1}\6^{-1}$.
\end{itemize}
Let $\6=x,\2=y$ and $\1=z$.
Then $\pi_1(X_{V_5})$ is isomorphic to $\A_n$. 
Therefore, for $n\geq5$ we have $g(\A_n)\leq5$.

Thus, the proof of (5) of Theorem~\ref{1.2} is completed.
\end{proof}

\subsection{Proof of (6) of Theorem~\ref{1.2}}

\begin{figure}
\subfigure[The loop $A_{i,j}$, $1\leq i<j\leq r$.]{\includegraphics{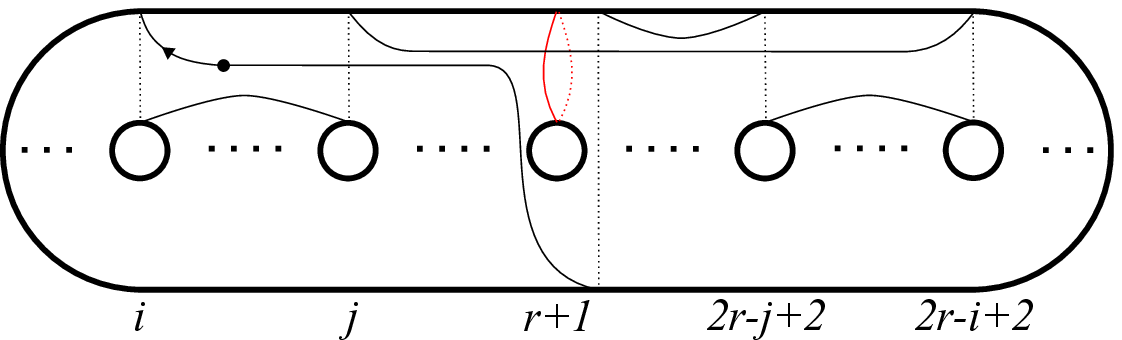}}

\subfigure[The loop $B_{i,j}$, $1\leq i<j\leq r$.]{\includegraphics{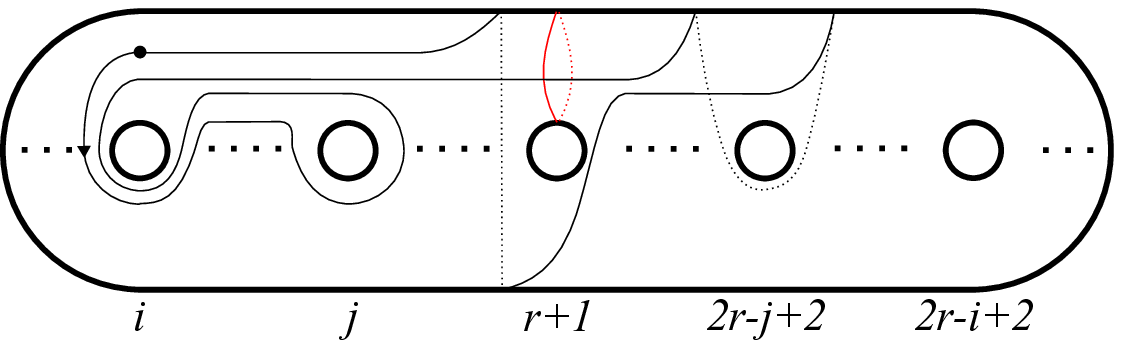}}

\subfigure[The loop $C_{i,j}$, $1\leq i<j\leq r$.]{\includegraphics{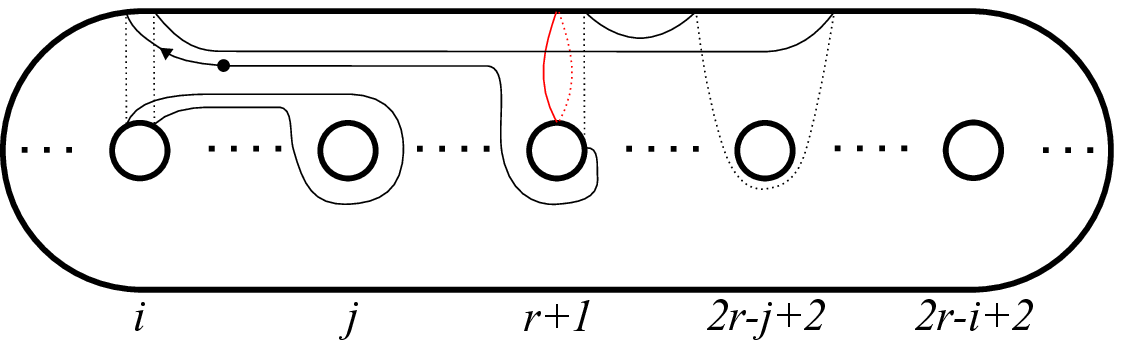}}

\subfigure[The loop $C_{i,j}$, $1\leq j<i\leq r$.]{\includegraphics{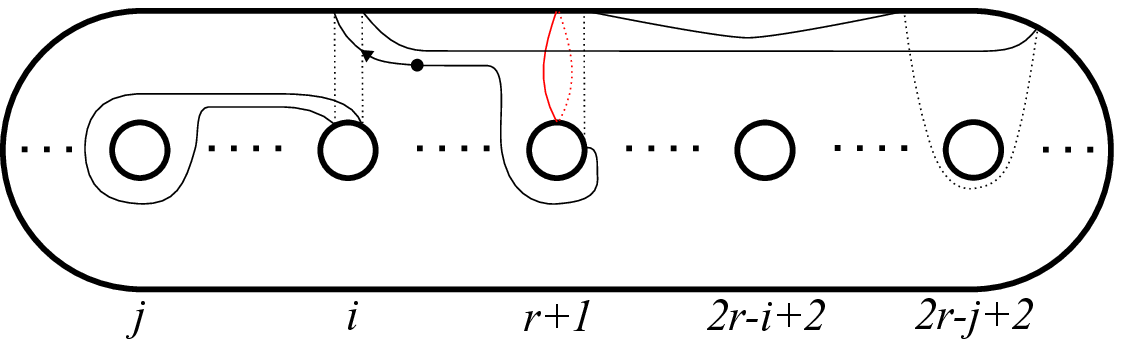}}

\subfigure[The loop $C_{i,i}$, $1\leq i\leq r$.]{\includegraphics{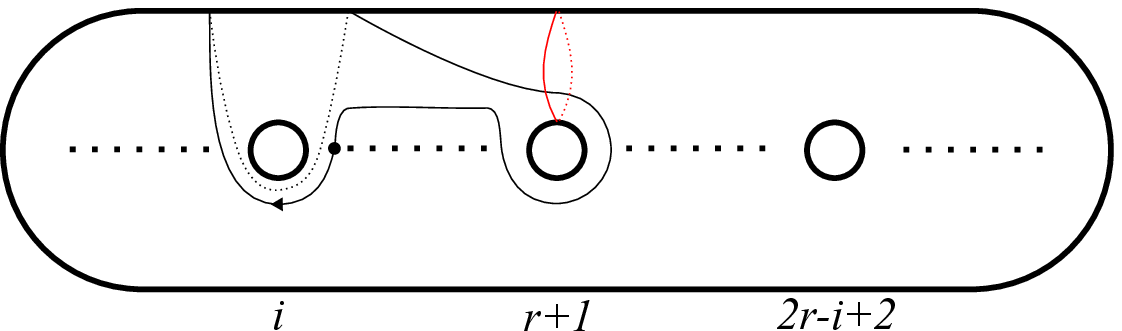}}
\caption{}\label{(6)1}
\end{figure}

\begin{figure}
\subfigure[The loop $R_i^{m_i}$ with $m_i=2$, $1\leq i\leq r$.]{\includegraphics{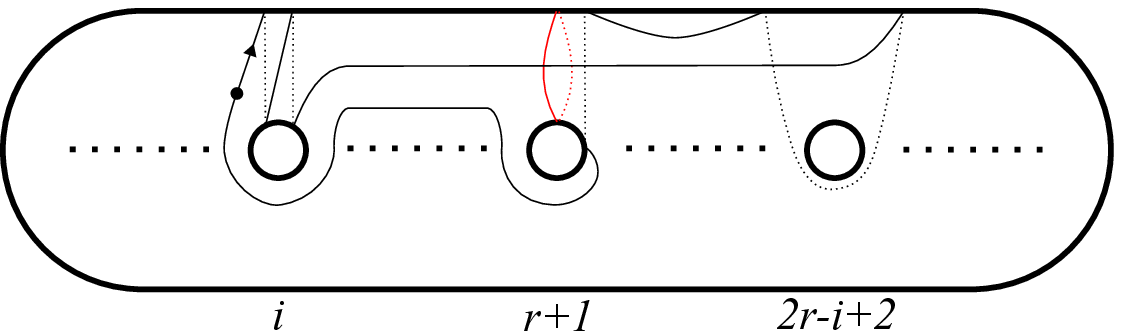}}

\subfigure[The loop $R_{r+i}^{m_{r+i}}$ with $m_{r+i}=2$, $1\leq i\leq r$.]{\includegraphics{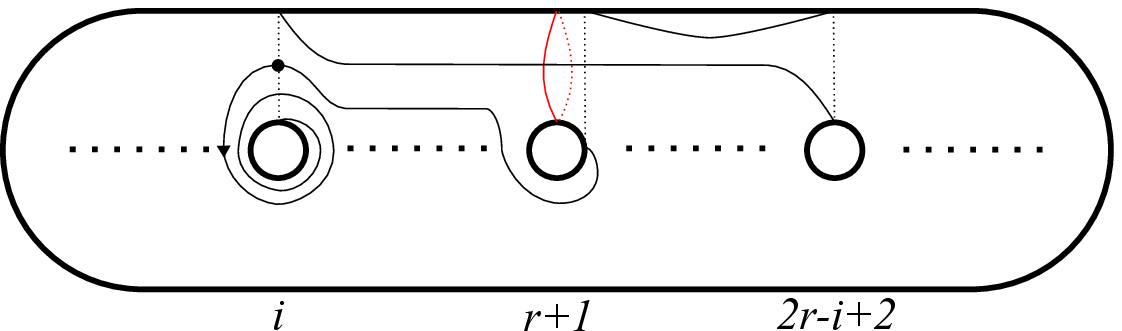}}
\caption{}\label{(6)2}
\end{figure}

\begin{figure}
\subfigure[The loop $A_{i,j}$, $1\leq i<j\leq r$.]{\includegraphics{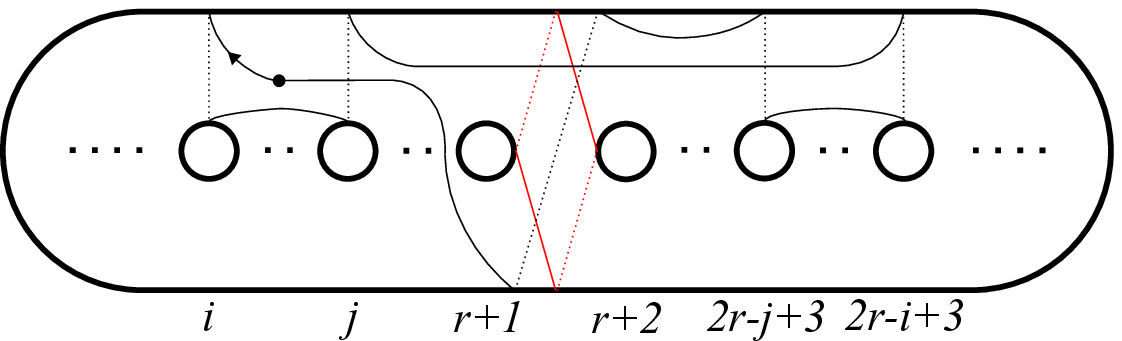}}

\subfigure[The loop $B_{i,j}$, $1\leq i<j\leq r$.]{\includegraphics{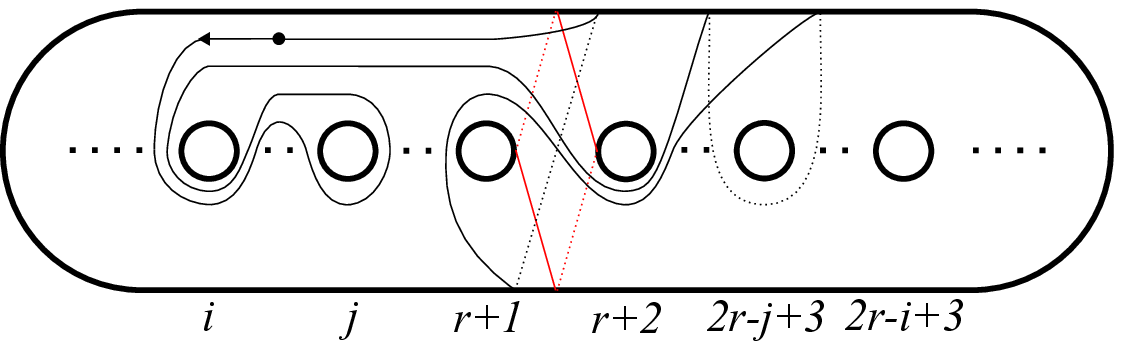}}

\subfigure[The loop $C_{i,j}$, $1\leq i<j\leq r$.]{\includegraphics{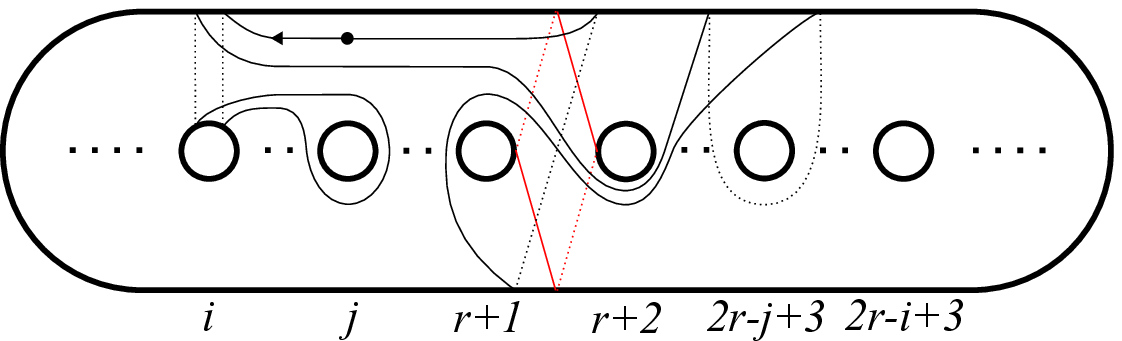}}

\subfigure[The loop $C_{i,j}$, $1\leq j<i\leq r$.]{\includegraphics{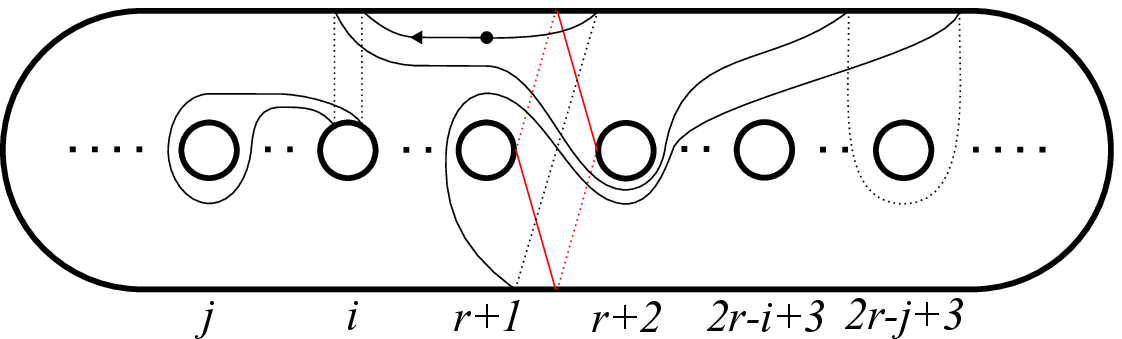}}

\subfigure[The loop $C_{i,i}$, $1\leq i\leq r$.]{\includegraphics{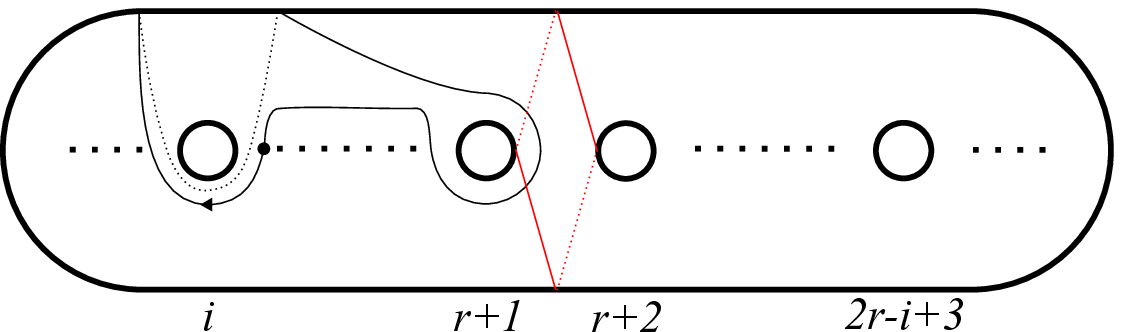}}
\caption{}\label{(6)3}
\end{figure}

\begin{figure}
\subfigure[The loop $A_{i,r+1}$, $1\leq i\leq r$.]{\includegraphics{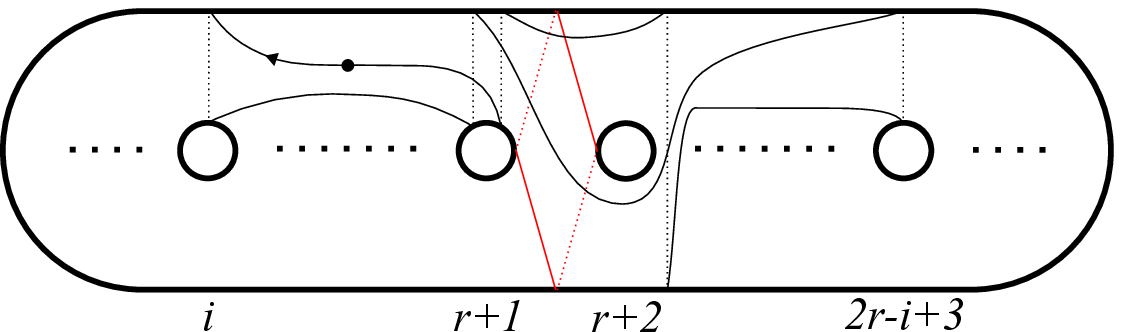}}

\subfigure[The loop $C_{r+1,i}$, $1\leq i\leq r$.]{\includegraphics{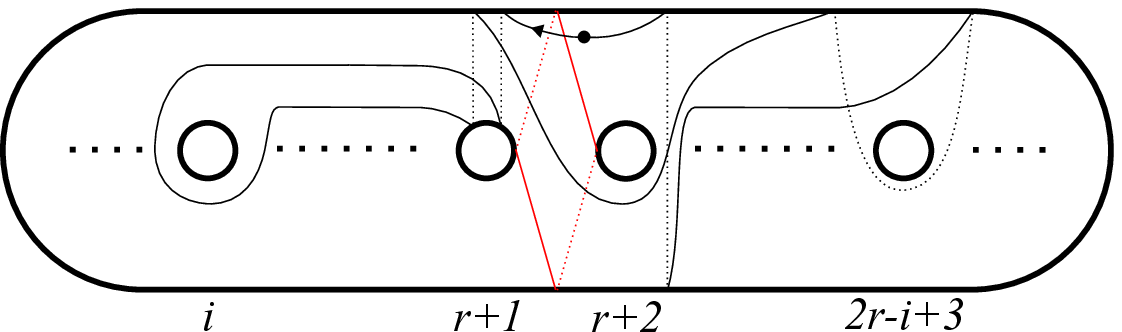}}
\caption{}\label{(6)4}
\end{figure}

\begin{figure}
\subfigure[The loop $R_i^{m_i}$ with $m_i=2$, $1\leq i\leq r$.]{\includegraphics{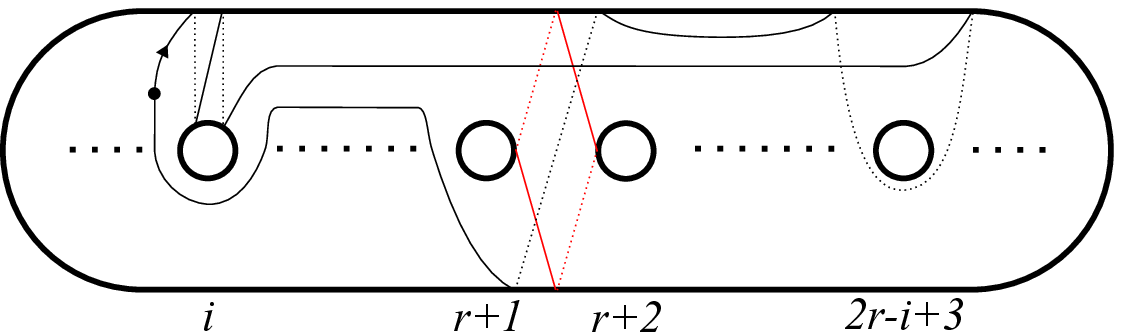}}

\subfigure[The loop $R_{r+i}^{m_{r+i}}$ with $m_i=2$, $1\leq i\leq r$.]{\includegraphics{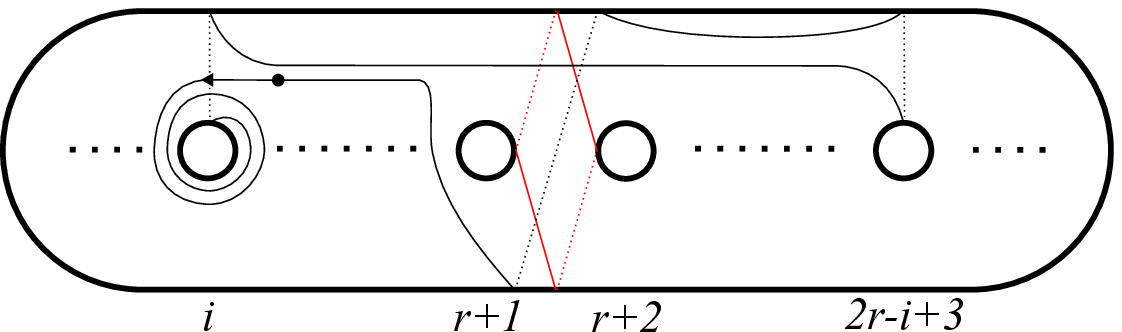}}

\subfigure[The loop $R_{2r+1}^{m_{2r+1}}$ with $m_{2r+1}=2$.]{\includegraphics{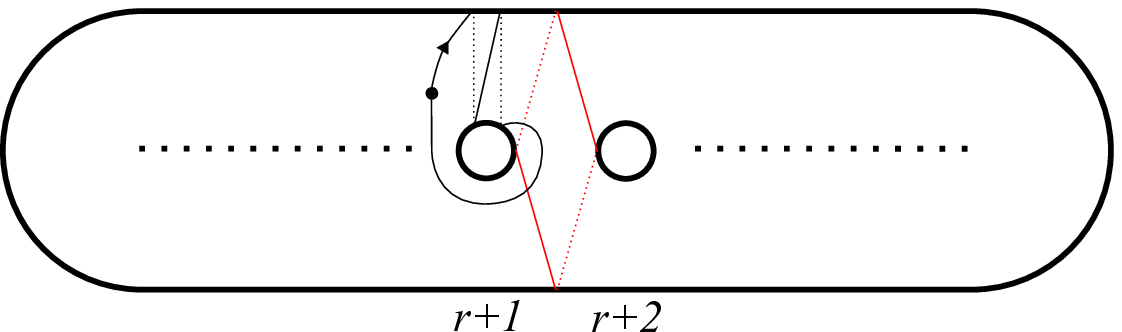}}
\caption{}\label{(6)5}
\end{figure}

\begin{proof}[Proof of (6) of Theorem~\ref{1.2}]
Let $n,k\geq0$ be integers with $n+k\geq3$.

At first, we consider the case $n+k$ is even.
We put $n+k=2r$.
Let $A_{i,j}$ and $B_{i,j}$ be simple closed curves on $\Sigma_{n+k+1}$
 as shown in (a) and (b) of Figure~\ref{(6)1}, respectively, where 
$1\leq{}i<j\leq{}r$, and let $C_{i,j}$ be the simple closed curve on
 $\Sigma_{n+k+1}$ as shown in (c), (d) and (e) of Figure~\ref{(6)1}, where 
$1\leq{}i,j\leq{}r$.
Note that each of $A_{i,j}$, $B_{i,j}$ and $C_{i,j}$ intersects $a_{r+1}$
 transversely at only one point.
Loops $A_{i,j}$, $B_{i,j}$ and $C_{i,j}$ can be described in
 $\pi_1(\Sigma_{n+k+1})$, up to conjugation, as follows
\begin{itemize}
 \item $A_{i,j}=a_ia_j^{-1}a_{2r-i+2}a_{2r-j+2}^{-1}(c_{r+1}^{-1}b_{r+1}^{-1})$,
       where $1\leq{}i<j\leq{}r$,
 \item $B_{i,j}=b_ib_jb_i^{-1}a_{2r-j+2}b_{2r-j+2}a_{2r-j+2}^{-1}(b_{r+1}^{-1}c_r)$,
       where $1\leq{}i<j\leq{}r$,
 \item $C_{i,j}=a_ib_j^{-1}a_i^{-1}a_{2r-j+2}b_{2r-j+2}^{-1}a_{2r-j+2}^{-1}(a_{r+1}b_{r+1}^{-1})$,
       where $1\leq{}i,j\leq{}r$ and $i\neq{}j$,
 \item $C_{i,i}=b_i^{-1}a_ib_ia_i^{-1}(b_{r+1}^{-1})$, 
       where $1\leq{}i\leq{}r$.
\end{itemize}
Let $V_6$ be the following:
$$
V_6=W(\prod_{1\leq{}i<j\leq{}r}W^{t_{A_{i,j}}})(\prod_{1\leq{}i<j\leq{}r}W^{t_{B_{i,j}}})(\prod_{1\leq{}i,j\leq{}r}W^{t_{C_{i,j}}}).
$$
Note that we have relations $a_{r+1}=1$, $b_{r+1}=1$, $c_r=1$ and
 $c_{r+1}=1$ in $\pi_1(X_W)$.
In addition, we have the relation
 $a_{2r-j+2}b_{2r-j+2}a_{2r-j+2}^{-1}=b_j^{-1}$ in $\pi_1(X_W)$ (see
 the presentation of $\pi_1(X_W)$ in the proof of
 Proposition~\ref{3.1}).
Then, from Proposition~\ref{2.2}, the fundamental group $\pi_1(X_{V_6})$
 can be presented with generators $a_1,b_1,\dots,a_r,b_r$ and with
 relations
\begin{itemize}
 \item $a_ia_j^{-1}a_i^{-1}a_j$, where $1\leq{}i<j\leq{}r$,
 \item $b_ib_jb_i^{-1}b_j^{-1}$, where $1\leq{}i<j\leq{}r$,
 \item $a_ib_j^{-1}a_i^{-1}b_j$, where $1\leq{}i,j\leq{}r$ and $i\neq{}j$,
 \item $b_i^{-1}a_ib_ia_i^{-1}$, where $1\leq{}i\leq{}r$.
\end{itemize}
Namely, $\pi_1(X_{V_6})$ is isomorphic to $\Z^{2r}$.
We next consider the simple closed curve $R_i^{m_i}$ on $\Sigma_{n+k+1}$
 as shown in Figure~\ref{(6)2}, where $1\leq{}i\leq2r$ and $m_i\geq2$.
Note that $R_i^{m_i}$ intersects $a_{r+1}$ transversely at only one
 point.
Loops $R_i^{m_i}$ can be described in $\pi_1(\Sigma_{n+k+1})$, up to
 conjugation, as follows
\begin{itemize}
 \item $R_i^{m_i}=a_i^{m_i}(a_{2r-i+2}b_{2r-i+2}^{-1}a_{2r-i+2}^{-1}a_{r+1}b_{r+1}^{-1}b_i^{-1})$,
       where $1\leq{}i\leq{}r$,
 \item $R_{r+i}^{m_{r+i}}=b_i^{m_{r+i}}(a_i^{-1}a_{2r-i+2}^{-1}a_{r+1}b_{r+1}^{-1})$,
       where $1\leq{}i\leq{}r$.
\end{itemize}
Let $V_7$ be the following:
$$
V_7=V_6(\prod_{1\leq{}i\leq{}k}W^{t_{R_i^{m_i}}}).
$$
Then, from Proposition~\ref{2.2}, the fundamental group $\pi_1(X_{V_7})$
 is isomorphic to $\Z^n\oplus\Z_{m_1}\oplus\cdots\oplus\Z_{m_k}$.
Therefore, if $n+k$ is even, we have 
$g(\Z^n\oplus\Z_{m_1}\oplus\cdots\oplus\Z_{m_k})\leq n+k+1$.

Next, we consider the case $n+k$ is odd.
We put $n+k=2r+1$.
Let $A_{i,j}$ and $B_{i,j}$ be simple closed curves on $\Sigma_{n+k+1}$
 as shown in (a) and (b) of Figure~\ref{(6)3}, respectively, where 
$1\leq{}i<j\leq{}r$, and let $C_{i,j}$ be the simple closed curve on
 $\Sigma_{n+k+1}$ as shown in (c), (d) and (e) of Figure~\ref{(6)3}, where 
$1\leq{}i,j\leq{}r$.
In addition, let $A_{i,r+1}$ and $C_{r+1,i}$ be simple closed curves on
 $\Sigma_{n+k+1}$ as shown in (a) and (b) of Figure~\ref{(6)4}, where 
$1\leq{}i\leq{}r$.
Note that each of $A_{i,j}$, $B_{i,j}$ and $C_{i,j}$ intersects $B_{2r+2}$
 transversely at only one point.
Loops $A_{i,j}$, $B_{i,j}$ and $C_{i,j}$ can be described in
 $\pi_1(\Sigma_{n+k+1})$, up to conjugation, as follows
\begin{itemize}
 \item $A_{i,j}=a_ia_j^{-1}a_{2r-i+3}a_{2r-j+3}^{-1}(c_{r+1}^{-1}b_{r+1}^{-1})$,
       where $1\leq{}i<j\leq{}r$,
 \item $A_{i,r+1}=a_ia_{r+1}^{-1}(b_{r+2})a_{2r-i+3}(c_{r+2})a_{r+1}$,
       where $1\leq{}i\leq{}r$,
 \item $B_{i,j}=b_ib_jb_i^{-1}(b_{r+2})a_{2r-j+3}b_{2r-j+3}a_{2r-j+3}^{-1}(b_{r+2}^{-1}b_{r+1}c_{r+1})$,
       where $1\leq{}i<j\leq{}r$,
 \item $C_{i,j}=a_ib_ja_i^{-1}(b_{r+2})a_{2r-j+3}b_{2r-j+3}a_{2r-j+3}^{-1}(b_{r+2}^{-1}b_{r+1}c_{r+1})$,
       where $1\leq{}i,j\leq{}r$ and $i\neq{}j$,
 \item $C_{i,i}=b_i^{-1}a_ib_ia_i^{-1}(b_{r+1}^{-1})$, 
       where $1\leq{}i\leq{}r$,
 \item $C_{r+1,i}=a_{r+1}b_ia_{r+1}^{-1}(b_{r+2})a_{2r-i+3}b_{2r-i+3}a_{2r-i+3}^{-1}(c_{r+2})$,
       where $1\leq{}i\leq{}r$.
\end{itemize}
Let $V_8$ be the following:
$$
V_8=WW^{t_{b_{r+1}}}(\prod_{1\leq{}i<j\leq{}r+1}W^{t_{A_{i,j}}})(\prod_{1\leq{}i<j\leq{}r}W^{t_{B_{i,j}}})(\prod_{1\leq{}i\leq{}r+1,1\leq{}j\leq{}r}W^{t_{C_{i,j}}}).
$$
Since $b_{r+1}$ intersects $B_{2r+2}$ transversely at only one point, we
 have the relation $b_{r+1}=1$ in $\pi_1(X_{WW^{t_{b_{r+1}}}})$ from
 Proposition~\ref{2.2}.
Hence we have relations $b_{r+2}=1$ and $c_{r+2}=1$ in
 $\pi_1(X_{WW^{t_{b_{r+1}}}})$.
Then, from Proposition~\ref{2.2} and the presentation of $\pi_1(X_W)$ in
 the proof of Proposition~\ref{3.1}, the fundamental group $\pi_1(X_{V_8})$
 is isomorphic to an abelian generated by $a_1,b_1,\dots,a_r,b_r$ and
 $a_{r+1}$.
We next consider the simple closed curve $R_i^{m_i}$ on $\Sigma_{n+k+1}$ as
 shown in Figure~\ref{(6)5}, where $1\leq{}i\leq2r+1$ and $m_i\geq2$.
Note that $R_i^{m_i}$ intersects $B_{2r+2}$ transversely at only one
 point.
Loops $R_i^{m_i}$ can be described in $\pi_1(\Sigma_{n+k+1})$, up to
 conjugation, as follows
\begin{itemize}
 \item $R_i^{m_i}=a_i^{m_i}(a_{2r-i+3}b_{2r-i+3}^{-1}a_{2r-i+3}^{-1}c_{r+1}^{-1}b_{r+1}^{-1}b_i^{-1})$,
       where $1\leq{}i\leq{}r$,
 \item $R_{r+i}^{m_{r+i}}=b_i^{m_{r+i}}(a_i^{-1}a_{2r-i+3}^{-1}c_{r+1}^{-1}b_{r+1}^{-1})$,
       where $1\leq{}i\leq{}r$,
 \item $R_{2r+1}^{m_{2r+1}}=a_{r+1}^{m_{2r+1}}(b_{r+1}^{-1})$.
\end{itemize}
Let $V_9$ be the following:
$$
V_9=V_8(\prod_{1\leq{}i\leq{}k}W^{t_{R_i^{m_i}}}).
$$
Then, from Proposition~\ref{2.2}, the fundamental group $\pi_1(X_{V_9})$
 is isomorphic to $\Z^n\oplus\Z_{m_1}\oplus\cdots\oplus\Z_{m_k}$.
Therefore, if $n+k$ is odd, we have 
$g(\Z^n\oplus\Z_{m_1}\oplus\cdots\oplus\Z_{m_k})\leq n+k+1$.

Moreover, it is immediately follows from Theorem~\ref{2.3}~(2) or (5)
(cf. \cite{k2}) that 
$g(\Z^n\oplus\Z_{m_1}\oplus\cdots\oplus\Z_{m_k})\geq\frac{n+k+1}{2}$. 
Thus, the proof of (6) of Theorem~\ref{1.2} is completed.

\end{proof}

\appendix
\section{}
Theorem 5.1 of \cite{k2} stated followings.
\begin{itemize}
 \item $g(\Gamma)=0$ if and only if $\Gamma$ is the trivial group. 
 \item $g(\Gamma)=1$ if and only if $\Gamma$ is isomorphic to
       $\Z\oplus\Z$.
 \item $g(\Gamma)=2$ if $\Gamma$ is isomorphic to $\Z$, $\Z\oplus\Z_n$,
       $\Z_n\oplus\Z_m$ or $\Z_n$, where $n,m\geq2$.
\end{itemize}
However, we have that $g(\Z)=g(\Z\oplus\Z_n)=1$.
In fact, $S^3\times{}S^1$, the product of Hopf fibration with $S^1$, is
the genus-$1$ Lefschetz fibration without singular fibers whose
fundamental group is isomorphic to $\Z$.
In addition, $L(n,1)\times{}S^1$ is the genus-$1$ Lefschetz fibration
without singular fibers whose fundamental group is isomorphic to 
$\Z\oplus\Z_n$, where $L(n,1)$ is the lens space of the type $(n,1)$.

For integers $n$ and $m$, let
$\varphi_{n,m}:\partial{}D^2\times{}T^2\to\partial{}D^2{}\times{}T^2$ be
$\varphi_{n,m}(e^{\alpha{}i},e^{\beta{}i},e^{\gamma{}i})=(e^{\alpha{}i},e^{(\beta+n\alpha)i},e^{(\gamma+m\alpha)i})$,
and let $X_{n,m}$ be
$X_{n,m}=D^2\times{}T^2\cup_{\varphi_{n,m}}D^2\times{}T^2$, where $D^2$
is a disk and $T^2$ is a torus.
Then $X_{n,m}$ is a $T^2$-bundle over $S^2$.
Conversely any $T^2$-bundle over $S^2$ is isomorphic to some $X_{n,m}$
as a bundle.
For example, $X_{0,0}=S^2\times{}T^2$, $X_{1,0}=S^3\times{}S^1$ and
$X_{n,0}=L(n,1)\times{}S^1$.
Let $d$ be the greatest common divisor of $n$ with $m$, then we have 
$\pi_1(X_{n,m})$ is isomorphic to $\Z\oplus\Z_d$, where we suppose that
the greatest common divisor of $0$ with $0$ is $0$.

Therefore the fundamental group of a genus-$1$ Lefschetz fibration
without singular fibers is $\Z\oplus\Z$, $\Z$ or $\Z\oplus\Z_n$ for some
$n\geq2$.
On the other hand, a genus-$1$ Lefschetz fibration with singular fibers
is an elliptic surface $E(n)$ for some $n\geq1$ (see \cite{k} and
\cite{mo}), and $E(n)$ is simply connected.

We summarize:

\begin{thm}\label{5.1}
\begin{enumerate}
 \item $g(\Gamma)=0$ if and only if $\Gamma$ is the trivial group. 
 \item $g(\Gamma)=1$ if and only if $\Gamma$ is isomorphic to 
       $\Z\oplus\Z$, $\Z$ or $\Z\oplus\Z_n$ for $n\geq2$.
 \item $g(\Gamma)=2$ if $\Gamma$ is isomorphic to 
       $\Z_n$ or $\Z_n\oplus\Z_m$ for $n,m\geq2$.
\end{enumerate}
\end{thm}

\section*{Acknowledgement}

The author would like to express thanks to Susumu Hirose and Naoyuki
Monden for their valuable suggestions and useful comments.

\end{document}